\documentclass[reqno,a4paper,dvipsnames]{amsart}

\usepackage{amsmath, amsthm, amsfonts, amssymb, amsgen, amstext, amsbsy, amsopn, amscd}

\usepackage{MnSymbol}

\usepackage{pgfkeys}
\usepackage{mathtools}
\usepackage[left=1in,right=1in,bottom=1.5in]{geometry}
\usepackage{enumerate}
\usepackage{relsize}
\usepackage{pst-node}

\usepackage{xcolor}
\usepackage{color}
\colorlet{linkequation}{BrickRed}

\usepackage[colorlinks=true,linkcolor=blue]{hyperref}

\usepackage{cleveref}
\usepackage{tikz-cd}
\usepackage{eqnarray,amsmath}

\def\multiset#1#2{\ensuremath{\left(\kern-.3em\left(\genfrac{}{}{0pt}{}{#1}{#2}\right)\kern-.3em\right)}}

\usetikzlibrary{arrows}

 \usepackage{latexsym}
 \usepackage{graphics}

\usepackage{lastpage}
\usepackage{fancyhdr}
\usepackage{multirow}

\allowdisplaybreaks
\usepackage{graphicx}
\graphicspath{ {F:/IMAGES/} }
	\definecolor{americanrose}{rgb}{1.0, 0.01, 0.24}
    	\definecolor{amethyst}{rgb}{0.6, 0.4, 0.8}
        \definecolor{antiquefuchsia}{rgb}{0.57, 0.36, 0.51}
        \definecolor{byzantium}{rgb}{0.44, 0.16, 0.39}
        	\definecolor{BrickRed}{rgb}{0.43, 0.21, 0.1}

\makeatletter
\def\oversortoftilde#1{\mathop{\vbox{\m@th\ialign{##\crcr\noalign{\kern3\p@}%
      \sortoftildefill\crcr\noalign{\kern3\p@\nointerlineskip}%
      $\hfil\displaystyle{#1}\hfil$\crcr}}}\limits}

\def\sortoftildefill{$\m@th \setbox\z@\hbox{$\braceld$}
  \braceld\leaders\vrule \@height\ht\z@ \@depth\z@\hfill\braceru$}

\makeatother

 \newcommand{\m}{\mathfrak{m} }
  \newcommand{\n}{\mathfrak{n} }
 \newcommand{\q}{\mathfrak{q} }

\newcommand{\kk}{\mathbb{K}}

 \newcommand{\M}{\mathfrak{M}}

 \newcommand{\grade}{\operatorname{grade}}
 \newcommand{\depth}{\operatorname{depth}}

 \newcommand{\Hom}{\operatorname{Hom}}

\newcommand{\proset}{\,\mathrel{\lower 4pt\hbox{$\scriptscriptstyle/$}
\mkern -14mu\subseteq }\,} %for proper subset

 \newtheorem{theorem}{Theorem}[section]
 \newtheorem{corollary}[theorem]{Corollary}
 \newtheorem{lemma}[theorem]{Lemma}
 \newtheorem{proposition}[theorem]{Proposition}

\usepackage{amsmath}
\newtheorem{notation}[theorem]{Notation}

 \theoremstyle{definition}
 
 \newtheorem{remark}[theorem]{Remark}
 \newtheorem{definition}[theorem]{Definition}

 \newtheorem{example}[theorem]{Example}

\allowdisplaybreaks

\title[Bounds On the Second Hilbert Coefficient and the Depth of the Associated Graded Ring] {Bounds On the Second Hilbert Coefficient and the Depth of the Associated Graded Ring}

\author{Clare D'Cruz, Mousumi Mandal and Shruti Priya }
\date{}

\thanks{AMS Classification 2020: 13A30, 13D40, 13E05, 13H10.}
\thanks{Key words and phrases: sectional genus, second Hilbert coefficients, parameter ideals, $S_{2}$-fication, associated graded ring .}

\address{Chennai Mathematical Institute, Plot H1 SIPCOT IT Park, Siruseri, Kelambakkam 603103,
Tamil Nadu, India} \email{clare@cmi.ac.in}
\address{Indian Institute of Technology Kharagpur, Kharagpur 721302, West Bengal, India} \email{mousumi@maths.iitkgp.ac.in}
\address{ Indian Institute of Technology Kharagpur, Kharagpur 721302, West Bengal, India} \email{shruti96312@kgpian.iitkgp.ac.in}
\begin{document}
\allowdisplaybreaks
\begin{abstract}
Let $(R, \m)$ be a  Noetherian local ring    of dimension $d \geq 1$ with $\depth R \geq d-1,$ and let $I$ be an $\m$-primary ideal. In this paper, 
we study bounds on the second Hilbert coefficient of $I$, denoted  by  $e_{2}(I)$.  Under the assumption that the associated graded ring $G(I)$ has depth at least $d-1,$ we first establish a lower bound for $e_{2}(I).$ We then extend several known results   
 from the Cohen-Macaulay case to this general setting and  obtain upper bounds for $e_{2}(I)$ in terms of the sectional genus denoted by $\mathrm{g}_{s}(I)$ and the Hilbert coefficients of $I$ and those of a minimal reduction $Q$ of $I$. 
 We further analyze the extremal case when $e_{2}(I)$ attains this bound and relate it to the  depth of $G(I)$.  
 In addition, for Buchsbaum local rings, we establish a sharp upper bound for $e_{2}(\m)$ using the technique of $S_{2}$-fication. 
 Finally, in the Cohen-Macaulay case, we give sufficient conditions to ensure good properties on the depth of $G(I)$ and  of $G(I^n)$ under the assumption that $e_{2}(I)=0$.
\end{abstract}
\maketitle
\section{Introduction}
Throughout this paper, $(R,\m)$  will denote   a Noetherian local ring of positive  dimension $d$ with infinite residue field $k=R/\m$ and $I$ will be  an $\m$-primary ideal  in $R$. In the past few decades, Hilbert coefficients $e_i(I)$ ($0 \leq i \leq d$) have played an important role in understanding  properties of the blowup algebras, namely,  the Rees algebra of $I$ denoted by $\mathcal{R}(I):=\displaystyle \bigoplus_{n \geq 0}I^{n}t^{n}$ (here $t$ is an indeterminate) and the associated graded ring of $I$ denoted by $G(I):=\displaystyle \bigoplus_{n \geq 0}I^{n}/I^{n+1}$.

Let   $(R, \m)$ be a Cohen-Macaulay ring. One of the first results relating the Hilbert coefficients was given  by Northcott \cite{northcott}. He proved that $e_{1}(I) \geq e_{0}(I)-\lambda(R/I).$
Later, Huneke \cite[Theorem 2.1]{huneke}, and Ooishi \cite{ooishi}  independently proved that Northcott's  inequality is an equality if and only if $I^2=QI$ for any minimal reduction $Q$ of $I$.  
Furthermore,  $G(I)$ is Cohen-Macaulay and if $d \geq 2$, then $\mathcal{R}(I)$ is Cohen-Macaulay. 
Motivated by these results, several researchers obtained bounds for $e_{1}(I)$ in terms of $e_{0}(I)$ 
(see \cite{elias 2}, \cite{elias 3}, \cite{krishna}, \cite{km}, \cite{rossi}, \cite{rv}). 
Rhodes was the first to give an upper bound for $e_2(I)$ in terms of $e_1(I)$ (see \cite[Proposition 6.1 (iv)]{rhodes}), where he proved that $e_{2}(I) \leq \binom{e_{1}(I)+1}{2}.$
Later  Narita \cite{narita}  proved that $e_{2}(I) \geq 0$. 
   Kirby and  Mehran in \cite{km} obtained some interesting results on the Hilbert coefficients  for $d \leq 2$ and as a consequence reproved Rhodes' result on $e_2(I)$. Recently, in \cite[Theorem 5.1]{dung},  Dung, Elias, and Hoa  gave another upper bound for $e_2(I)$. They showed that $e_{2}(I) \leq \binom{e_{0}(I)-b+1}{3}$, where $b$ is a positive integer such that $I \subseteq \m^{b}$.

When   $(R, \m)$  is  not Cohen-Macaulay,  obtaining  bounds for $e_i(I)$  is significantly more difficult. 
The work of Vasconcelos \cite{vasconcelos}  on $e_1(I)$,    which he called the Chern number, revived interest in studying the Hilbert coefficients. 
 In \cite[Theorem 3.2]{gotoozeki}, Goto and Ozeki proved that if $\dim R=2$ and 
 $\depth R>0$, then   for any parameter ideal $Q$ we have   $e_{2}(Q) \leq 0$. Later, McCune extended this result and showed that if $Q$ is a parameter ideal in a Noetherian local ring of dimension $d \geq 2$ with depth at least $d-1$,  then $e_2(Q) \leq0$. Further, she proved that $e_{2}(Q)=0$ if and only if $\depth G(Q) \geq d-1 $ and $n(Q):=\min \{k: H_Q(n) =P_Q(n) \text{ for all } n > k\} \geq 2-d$ \cite[Theorem 3.5 (ii)]{lori}. 
Moreover,  if  $e_2(Q)=0$, then $e_i(Q)=0$ for all $i \geq 2$ (see \cite[Theorem 3.5 (iii)]{lori}).
Thus McCune's result highlights the importance of understanding the vanishing of $e_2(Q)$.
  Recently, in   \cite[Theorem 3.1]{vdtrung}, Trung  proved that if  $d\geq 2$ and  $\depth R \geq d-1$,  then   $e_{2}(I) \leq \displaystyle \sum_{n \geq 1}n \lambda(I^{n+1}/QI^{n}),$ where $Q$ is a minimal reduction of $I.$   In the case  $R$ is a Buchsbaum local ring,  a lower bound for $e_{2}(I)$ was obtained by Goto and Ozeki  in \cite[Theorem 2.11]{sk}. 
 
  These results motivated us to investigate   bounds for $e_2(I)$  in the case $I$ is an $\m$-primary ideal  in a Noetherian local ring of dimension $d \geq 2$.  
The  objective of this paper is twofold. In the first part, we  establish bounds for the second Hilbert coefficient  of an $\m$-primary ideal $I$ in  Noetherian local ring of dimension $d \geq 2$ in terms of the \textit{sectional genus}  of $I$.  The notion of sectional genus of an ideal $I$, denoted by $\mathrm{g}_{s}(I)$ was introduced by Ooishi in \cite{ooishi},  and is defined as:
 \begin{center}
     $\mathrm{g}_{s}(I)=\lambda(R/I)-e_{0}(I)+e_{1}(I).$
 \end{center} 
 For an ideal $I$, the sectional genus is  an important invariant associated with it. In Cohen-Macaulay local rings, it is closely related to  the structure of the associated graded ring and the Hilbert coefficients. For example, if $\mathrm{g}_{s}(I)=0,$ then $G(I)$ is Cohen-Macaulay. On the other hand, in the  non-Cohen–Macaulay case, very little  is known about the behavior and significance of sectional genus.
In the second part, assuming that $R$ is Cohen-Macaulay, we investigate how the vanishing of $e_{2}(I)$  yields interesting results on the depth of the associated graded ring of $I.$  

We now discuss the organization of the paper and  summarize  the main results.  The paper is organized into six sections.
Section \ref{section2} is devoted to definitions, preliminary concepts, and notations.
 In this section, we also  extend Huckaba's formulas for the Hilbert coefficients 
 \cite[Corollary 2.10]{sam} to  local rings $R$ of dimension $d \geq 2$ with $\depth R \geq d-1$ (see Corollary \ref{e_i}). This result will play an important role  in the subsequent sections.

In Section \ref{section3}, we establish  a lower bound for $e_2(I).$ From a result of Rossi and Valla (see \cite[Theorem~3.1(a)]{rv}), it follows that  if $I$ is an $\m$-primary ideal in a two-dimensional Cohen-Macaulay ring, then $e_2(I) \geq e_{1
}(I)-e_{0}(I)+ \lambda(R/\widetilde{I}),$ where $\widetilde{I}$ is the \textit{Ratliff-Rush closure} of $I$ (see Definition \ref{ratliffrushdef}). For $d \ge 2$, under the assumption that $R$ is Cohen-Macaulay and $\depth ~G(I) \geq d-1$,    a lower bound for $e_2(I)$ was obtained  by Marley \cite[Corollary~2(3)]{marley}.  We obtain a lower bound for $e_2(I)$   for any $\m$-primary  ideal $I$ in a Noetherian ring  of dimension $d \ge 2$ in terms of the sectional genus of $I$, $e_0(I)$ and $e_1(I)$. We also recover Marley's result. We prove:
\begin{proposition} (Proposition \ref{newandfinalproofoflowerbound})
      Let $(R,\m)$ be a  local ring of dimension $d\geq 2$ and let $I$ be an $\m$-primary ideal.  If $\depth G(I) \geq d-1,$ then we have the following:
 \begin{enumerate} [\normalfont (i)]
     \item 
      For any minimal reduction $Q$ of $I$, 
      \begin{equation*} 
      e_{2}(I) \geq  \mathrm{g}_{s}(I)-\mathrm{g}_{s}(Q)+e_{1}(Q)\dbinom{e_{0}(I)-e_{1}(I)}{2}.
       \end{equation*}
       
              \item 
      \cite[Corollary~2 (3)]{marley}
       If $R$ is Cohen-Macaulay, then 
       $e_{2}(I)\geq e_{1}(I)-e_{0}(I)+\lambda(R/I)$.    
 \end{enumerate}
 \end{proposition}

In  Section \ref{section4}, our main result is on the upper bound for $e_2(I)$ for an $\m$-primary ideal $I$ in  a Noetherian local ring $R$ of dimension $d \geq 2$, under the assumption that $\depth R \geq d-1$.  We first focus on $\m$-primary parameter ideals. In \cite[Theorem 3.5]{lori} McCune  proved that if $Q$ is a parameter ideal and $\depth R \geq d-1$, then $e_2(Q)$ is non-positive. By an example (see  \cite[Example 3.7]{lori}) she showed that the assumption on depth cannot be dropped. In this paper we give an alternate proof of McCune's result (\cite[Theorem 3.5]{lori}). We also give an upper bound for $e_2(Q)$ under the weaker assumption  that $\depth R \ge d-2$ (see Proposition \ref{parameterbounddepthzero}). 

We next study $e_2(I)$, where $I$ is an $\m$-primary ideal and $d \geq 2$. Since the reduction number of $I$ 
with respect to a minimal reduction $Q$, denoted by $r_Q(I)$  plays a fundamental role in the course of our 
study on $e_2(I)$, we first obtain an upper bound for $r_Q(I)$. Several authors have obtained bounds for the 
reduction number (see Definition \ref{reductiondef}) of $\m$-primary ideals    in terms of the Hilbert 
coefficients.  In \cite{gghv}, the authors gave  a bound  for the reduction number of $\m$-primary ideals  
in two dimensional Buchsbaum ring. An upper bound for $r_Q(I)$  in terms of the co-length of $I$ and the Hilbert coefficients was obtained by Rossi in the case $R$ is Cohen-Macaulay and $d\le 2$ \cite[Corollary 1.5]{rossi1}.  Bounds for  the reduction number of  $\m$-primary ideals in a Cohen-Macualay ring have been given in   \cite{rv} and   \cite{mandal-saloni}. 
However, in the case when $R$ is not Cohen-Macaulay it is considerably  more challenging to obtain bounds for $r_Q(I)$. In this paper,   we  give an upper bound for $r_Q(I)$ in the case when $\depth R \ge d-1$ and $\depth G(I) \geq d-1$. This is a significant progress in this direction (see Lemma~\ref{2-prime}).

We then focus on upper bounds for $e_{2}(I).$ In \cite[Proposition 3.2]{ozeki}, Ozeki gave an upper bound 
for $e_2(I)$, for $\m$-primary ideals in a  Cohen–Macaulay  local ring in terms of the sectional genus.  In particular he proved that $e_{2}(I) \leq \binom{\mathrm{g}_{s}(I)+1}{2}.$  
We extend Ozeki's result on the upper bound of $e_2(I)$ to the case to the case when $\depth R \geq d-1$. We need to involve the sectional genus  $\mathrm{g}_s(Q)$ and $e_1(Q)$, where $Q$ is a minimal reduction of $I$. 

We also give upper and lower bounds for 
the higher Hilbert coefficients  when the upper bound for $e_2(I)$ is attained.
The main result in this section on $e_2(I)$ is:
\begin{theorem} \label{theoremA} (Theorem \ref{finalupperbounde2})
Let $(R,\m)$ be a Noetherian local ring of dimension $d\geq 2$ and  $\depth R \geq d-1.$  Then  for any $\m$-primary ideal  $I$   and a  minimal reduction $Q=(x_1,\ldots,x_{d})$ of $I$, 
\begin{equation} \label{introeq1}
    e_2(I) \leq 
    \binom{\mathrm{g}_s(I) - \mathrm{g}_s(Q) - (e_{0}(I)-e_{1}(Q))e_1(Q) + 1}{2}-\displaystyle \sum_{n=1}^{e_{0}(I)-e_{1}(Q)-1}
           n\lambda(U_{I}(d;n)).
\end{equation}
    
 Further, suppose equality holds in  (\ref{introeq1}), then the following are true: 
 \begin{enumerate}[\normalfont(i)]
     \item    $\depth G(I) \geq d-1.$
     \item    If $I$ is not a parameter ideal  then 
          $r_{Q}(I)=\mathrm{g}_s(I) - \mathrm{g}_s(Q) - (e_{0}(I)-e_{1}(Q))e_1(Q)+1.$
          \item   Set $t=\mathrm{g}_{s}(I)-\mathrm{g}_s(Q) - (e_{0}(I)-e_{1}(Q))e_1(Q).$ For $3 \leq i \leq d,$ we have 
   \[\displaystyle\binom{t + 1}{i}+e_{1}(Q)\binom{e_{0}(I)-e_{1}(Q)}{i} \leq e_{i}(I) \leq\displaystyle\binom{t + 1}{i}.\]
 \end{enumerate}
\end{theorem}

The differences  $e_i(I) - e_i(Q)$, where $I$ is $\m$-primary and $Q$ is a minimal reduction of $I$ is also of interest. This is motivated by  a result in \cite[Theorem 2.6]{ghezzihong}, where they study the deviation of $e_{1}(I)$ from $e_{1}(Q)$. As an immediate consequence of our main result we 
obtain an upper bound for the difference  $e_2(I)-e_2(Q)$, where $I$ is an $\m$-primary ideal in a Buchsbaum local ring of dimension $d \geq 2$ and $Q$ is a minimal reduction of $I$. 
More precisely, we 
 prove the following result:
 \begin{proposition}
     (Proposition \ref{deviationbound})
     Let $(R,\m)$ be a Buchsbaum local ring of dimension $d \geq 2,$ let $I$ be an $\m$-primary ideal with minimal reduction $Q.$ 
 Then 
 %Proposition \ref{deviationbound} implies that:
\begin{equation*}
 e_{2}(I)-e_{2}(Q) \leq \binom{\mathrm{g}_{s}(I)+(e_{0}(I)-e_{1}(Q))\left(1-e_{1}(Q)\right)}{2}.
\end{equation*}
 \end{proposition}
 
We also give an  infinite class of examples (see Example \ref{example depth 0}) which shows that the  this upper bound is sharp. 

In Section \ref{section5},  we establish another upper bound for $e_{2}(I)$ in a Buchsbaum local ring using the technique of $S_{2}$-fication.
Recall that a Noetherian local ring $(R, \m)$ is called a Buchsbaum ring if every system of parameters $x_{1},\ldots,x_{r}$ of $R$ is a weak sequence, i.e., 
     $(x_{1},\ldots,x_{i-1}):x_{i}=(x_{1},\ldots,x_{i-1}):\m \text{ for } i=1,\ldots, r.$
     Let $(R,\m)$ be a two-dimensional Buchsbaum local ring with positive depth,  and let $\mathtt{S}$ be  the $S_{2}$-fication of $R.$ For an $\m$-primary ideal $I$, assuming $I=I\mathtt{S},$ we derive an upper bound for $e_{2}(I)$ (see Proposition \ref{dc}). As a direct consequence,  we obtain a sharp upper bound for the second Hilbert coefficient of maximal ideal. More precisely, we prove the following:
\begin{theorem} (Theorem \ref{s2ficationmaximal})
  Let $(R,\m)$ be a Buchsbaum local ring of dimension $d \geq 2$ with $\depth R\geq d-1$. Let $Q$ be a minimal reduction of $\m$, then   $e_{2}(\m) \leq \dbinom{\mathrm{g}_{s}(\m)-e_{1}(Q)+1}{2}+e_{1}(Q).$
\end{theorem}

In Section \ref{section6}, we focus on the vanishing of  higher Hilbert coefficients  when $R$ is Cohen-Macaulay.  This has an important  consequence on the depth of $G(I)$. 
  
  Most of the relevant references can be found in the book by Rossi and Valla \cite{rv}.  
The vanishing of the Hilbert coefficients does imply good   depth  property for the associated graded ring of $I$ or some of  its powers. 
One of the first results in this direction was by  Narita \cite[Theorem 1]{narita}, who proved that for $d=2$, if $e_{2}(I)=0,$ then $G(I^{n})$ is Cohen-Macaulay for sufficiently large $n.$ 
Our goal is to characterize a sufficient conditions for  $\depth G(I)$ in terms of the Hilbert coefficients. 
We prove the following:
\begin{proposition} \label{prop1.2} (Proposition \ref{e2equalszeroCM})
   Let $(R, \m)$ be a Cohen-Macaulay local ring of dimension $d \geq 2,$ $I$ an $\m$-primary ideal and $\mathcal{I}=\{I_{n}\}$ be an $I$-admissible filtration. Suppose $e_{2}(\mathcal{I})=0,$ then $\depth G(\mathcal{I}) \neq d-1.$
\end{proposition}
As a consequence of Proposition  \ref{prop1.2}, we recover Narita's result  \cite[Theorem 1]{narita}, when $\mathcal{I}=\{I^{n}\}_{n \geq 0}$ (see Corollary \ref{core2equalszeroCM} (ii)). Note that Narita's  result does not extend to higher dimensions (see \cite[Example 6.3]{tjp2}), therefore, we obtain an analogue  of Narita's result \cite[Theorem 1]{narita} in higher dimension. In fact, we prove: 
\begin{proposition} \label{prop1.5}  (Proposition \ref{vanishinghilbcoeffCM})
       Let $(R, \m)$ be a Cohen-Macaulay local ring of dimension $d \geq 2,$ and $I$ an $\m$-primary ideal. If $e_{2}(I)=e_{3}(I)=\ldots=e_{d}(I)=0,$ then  the following hold:
       \begin{enumerate}[\normalfont(i)]
           \item  $\depth G(I) = 0$ or $d.$
           \item  $G(I^{n})$ is Cohen-Macaulay for sufficiently large $n.$
       \end{enumerate}
       \end{proposition}
 We conclude this section by proving a result on the sign of  $e_{d}(I)$ in a $d$-dimensional Cohen-Macaulay local ring, assuming  that $e_{i}(I)=0$ for all $2 \leq i \leq d-1$. In particular we prove the following:
  \begin{proposition} (Proposition \ref{edresult}) \label{prop1.6}
     Let $(R, \m)$ be a Cohen-Macaulay local ring of dimension $d \geq 2,$ and $I$ an $\m$-primary ideal. If $e_{2}(I)=e_{3}(I)=\ldots=e_{d-1}(I)=0,$ then $(-1)^{d}e_{d}(I) \geq 0.$
\end{proposition}

Recall that an m-primary ideal $I$ is called a \textit{generalized Narita ideal} if $e_{i}(I)=0$
 for every $2\leq i\leq d$. We remark that Part (2) of Proposition \ref{prop1.5} and Proposition \ref{prop1.6} are due to Puthenpurakal (see  \cite[Theorem 1.1 (3) and Theorem 1.4]{tjp3}, respectively). Since our proofs are based on different techniques, we include here for the sake of completeness.
 
We provide  explicit examples  which illustrate the  results obtained in  this section.

\section{Preliminaries}  \label{section2}
In this section  we recall  a few definitions needed in this paper. We also define some  notations and extend some well known results on the Hilbert coefficients to the case when  the ring is not Cohen-Macaulay.   For all undefined terms, we request the reader to refer  \cite{bh}.

  \begin{definition} \label{defhilbcoeff1}
       A sequence of ideals $\mathcal{I}=\{I_{n}\}_{n\in \mathbb Z}$ is called an $I$-\textit{admissible filtration} if for all $n, m \in \mathbb Z,$ 
       $(i)$ $I_{n+1} \subseteq I_{n},$ 
       $(ii)$ $I_{m}I_{n} \subseteq I_{m+n}$ and 
       $(iii)$  there exists $k \in \mathbb N$ such that $I^{n} \subseteq I_{n} \subseteq I^{n-k}$. The \textit{Hilbert-Samuel function} of $\mathcal{I}$ is defined as $H_{\mathcal{I}}(n)= \lambda(R/I_{n}),$ where $\lambda(-)$ denotes the length. For sufficiently large  $n,$ this function coincides with  a polynomial $P_{\mathcal{I}}(n) \in \mathbb Q[x]$ of degree $d,$  called the \textit{Hilbert-Samuel polynomial} of $\mathcal{I}$ and can be written  as:
\begin{equation*}
    P_{\mathcal{I}}(n)=e_{0}(\mathcal{I})\dbinom{n+d-1}{d}-e_{1}(\mathcal{I})\dbinom{n+d-2}{d-1}+\cdots+(-1)^{d}e_{d}(\mathcal{I}),
\end{equation*}
where  $e_{i}(\mathcal{I})$ for $i=0,\ldots,d$ are called  the  \textit{Hilbert coefficients} of  $\mathcal{I}$. The leading coefficient $e_{0}(\mathcal{I})$  is  the \textit{multiplicity} of $\mathcal{I}$.
 When $\mathcal{I}$ is the $I$-adic filtration $\{I^{n}\}_{n \geq 0}$, we write $H_{I}(n)$ and $P_{I}(n)$ for the corresponding function and polynomial.
  \end{definition}

  \begin{definition}
       The \textit{Hilbert series} $HS_{\mathcal{I}}(t)$ of an $I$-admissible filtration $\mathcal{I}=\{I_{n}\}_{n \in \mathbb Z}$  is the formal power series $\displaystyle \sum_{n \geq 0} \lambda(I_{n}/I_{n+1}) t^{n}$. 
    By the Hilbert-Serre theorem, we write 
    $ HS_{\mathcal{I}}(t)=\dfrac{h_{\mathcal{I}}(t)}{(1-t)^{d}}$,
where $h_{\mathcal{I}}(t) \in \mathbb{Z}[t]$ is the unique polynomial with $h_{\mathcal{I}}(1) \neq 0$, known as the \textit{h-polynomial} of $\mathcal{I}.$   The power series, $\displaystyle \sum_{n \geq 0} H_{\mathcal{I}}(n)t^{n}$ is called the \textit{Hilbert-Samuel series} of $\mathcal{I}.$ Note that $\displaystyle \sum_{n \geq 0} H_{\mathcal{I}}(n)t^{n}=\frac{h_{\mathcal{I}}(t)}{(1-t)^{d+1}},$
and from \cite[Page 9]{rv}, an easy computation shows that for all $i \geq 0,$ we have $e_{i}(\mathcal{I})=\dfrac{h^{(i)}_{\mathcal{I}}(1)}{i!},$
 where $h^{(i)}_{\mathcal{I}}(1)$ is the $i^{th}$-formal derivative of $h_{\mathcal{I}}(t)$ at $t=1.$
  \end{definition}

\begin{definition}
         An element $x\in I \backslash I^2$ is called a \textit{superficial element} of $I$ if there exists a positive integer $c$ such that $(I^{n+1}:x) \cap I^{c}=I^{n}$ for all $n \geq c.$  A sequence $x_{1},x_{2},\ldots,x_{j}$ is  a \textit{superficial sequence} of $I,$ if $x_{1}$ is a superficial element of $I$, and the image of  $x_{i}$  in $R/(x_{1},x_{2},\ldots,x_{i-1})$ is a superficial element of $I/(x_{1},x_{2},\ldots,x_{i-1})$ for all $i=2,\ldots,j.$  
\end{definition}

We recall some useful properties of superficial elements.

\begin{proposition}\cite[Proposition 1.2]{rv} \label{hibertcoeffsup}
Let $x$ be a superficial element of $I$ in a local ring $(R,\m)$. Then we have:
\begin{enumerate}[\normalfont(i)]
    \item \label{hibertcoeffsup-1}
    $\dim (R/(x))=d-1.$
    \item \label{hibertcoeffsup-2}
    $e_{i}(I/(x))=e_{i}(I)$ for every $i=0,\ldots,d-2.$
    \item \label{hibertcoeffsup-3}
    $e_{d-1}(I/(x))=e_{d-1}(I)+(-1)^{d-1}\lambda(0:x).$
    \item \label{hibertcoeffsup-4}
    There exists an integer $j$ such that for every $n \geq j-1,$ we have
    $$e_{d}(I/(x))=e_{d}(I)+(-1)^d\left[\displaystyle \sum_{i=0}^{n}\lambda\left(\frac{I^{i+1}:x}{I^{i}}\right)-(n+1)\lambda(0:x)\right].$$
    \item $x^*$ is a non-zero divisor in $G(I)$ if and only if $HS_{I}(t)=\dfrac{HS_{I/(x)}(t)}{(1-t)}$ if and only if $x$ is a non-zero divisor in $R$ and $e_{d}(I)=e_{d}(I/(x)).$
\end{enumerate}
\end{proposition}

We now recall an important result from the literature which provides formulas for all the Hilbert coefficients in a one-dimensional local ring.

\begin{lemma} \cite[Lemma 2.1, 2.2]{rv} \label{formulahilbcoeffindimone} Let $x$ be a superficial element of $I$ in a one dimensional local ring $(R,\m).$ Then for every $k \geq 1$ we have 
\begin{equation*}
    e_{k}(I)=\sum_{n \geq k-1}
    \binom{n}{k-1}\left[\lambda\left(\frac{I^{n+1}}{xI^{n}}\right)-\lambda(0:_{I^{n}}x)\right].
\end{equation*} 
\end{lemma}

\begin{definition} \label{ratliffrushdef}
    The  Ratliff-Rush closure of $I^n$ is the ideal 
$\widetilde{I^n}=\displaystyle \bigcup_{k \geq 1}(I^{n+k}:I^{k})$. 
 The \textit{Ratliff-Rush filtration} with respect to $I$ is the filtration
 $\mathcal{F}:=\{\widetilde{I^{n}}\}_{n \geq 0}.$ 
The associated graded ring of the Ratliff-Rush filtration is  $\widetilde{G}(I)=\displaystyle \bigoplus_{n \geq 0}\widetilde{I^{n}}/\widetilde{I^{n+1}}$. 
  If $\grade(I)>0$, then  $\widetilde{I^{n}}=I^{n}$ for sufficiently large $n$ and  hence,  
  $\widetilde{G}(I)$ is Noetherian. Moreover,  $\widetilde{e_{i}}(I) = e_i(I)$ for all $i=0, \ldots, d$.
  \end{definition}

\begin{definition} \label{reductiondef}
     An ideal $Q \subseteq I$ is a \textit{reduction} of $I$ if $QI^{n}=I^{n+1}$ for some $n \in \mathbb{N}$. We say that a reduction $Q$  
is a \textit{minimal reduction} if it is minimal with respect to inclusion among all reductions of $I$. The \textit{reduction number} of $I$ with respect to $Q$ is defined  as 
   $r_{Q}(I):=\min\{n: QI^{n}=I^{n+1}\}. $
\end{definition}

We state a few results on minimal reduction of ideals which will be used in our paper. 

\begin{remark} \label{remarksup} 
 From \cite[Theorem 14.14]{matsumura}, there exists a minimal reduction of $I$ generated by a system of parameters.  From   \cite[Lemma 1.2]{rtt} it follows that  every minimal reduction of $I$ can be generated by a superficial sequence of $I$.   Therefore, for an $\m$-primary ideal
$I$,  we can always choose   a superficial sequence $x_{1},\ldots,x_{d}$ which is also a system of parameters and such that  $Q = (x_{1},\ldots,x_{d})$ is a minimal reduction of $I$. Furthermore, if $\depth R = i,$  then we can choose a minimal reduction $Q=(x_1, \ldots, x_d)$ such that $x_{1},\ldots,x_{i}$ is a regular sequence in $R$ for $1\leq i \leq d$.
\end{remark}

 In \cite[Theorem 2.4 and Corollary 2.10]{sam}, Huckaba proved  results concerning the Hilbert coefficients of an ideal  in the case where $R$ is a Cohen-Macaulay local ring. We extend these results to the case $\depth R \geq d-1$
(see Proposition~\ref{hilbcoeffextension} and Corollary~\ref{e_i}),   as we need it in the subsequent sections.    We first fix a few notations.

\begin{notation} \normalfont
\begin{enumerate} [\normalfont(i)]
    \item  Let  $Q = (x_1, \ldots, x_d) $ be any minimal reduction of $I$ and let $n \geq 0$.
We define 
\begin{align*}
  l_{Q,I}(n) & :=\dfrac{I^{n+1}}{QI^{n}},&\\
  U_{I}(d;n) & := \frac{((x_{1}, \ldots, x_{d-1}):x_{d})\cap (I^{n}+(x_{1}, \ldots, x_{d-1}))}{(x_1, \ldots, x_{d-1})} \text { for all } d\geq 2.
\end{align*}

\item For an integer valued function $f:\mathbb Z \longrightarrow \mathbb Z$, the \textit{first difference} of $f$ is  $\Delta[f(n+1)]=f(n+1)-f(n)$ for all $n \in \mathbb Z.$ Inductively, we define the \textit{i-th difference function} of $f$ by $\Delta^{i}[f(n+1)]=\Delta^{i-1}[\Delta[f(n+1)]].$
\end{enumerate}
   \end{notation}

\begin{definition} \cite[Definition 2.1]{sam}
    Let $(R, \m)$ be a local ring of dimension $d >0$, and $I$ be an $\m$-primary ideal in $R$ with minimal reduction $Q$. Assume that  $Q=(x_{1}, \ldots, x_{d})$, where $x_{1},\ldots,x_{d}$ is a superficial sequence of $I$ and let $Q_{i}=(x_{1},\ldots,x_{i})$ for $0 \leq i \leq d$ where $Q_{0}=(0).$ For a non-negative integer $n$, define 
        \begin{align*}
        w_{n}(Q,I)
        = \begin{cases}
        0, & \mbox{ if } d=1 \\
     \displaystyle   \sum_{i=0}^{d-2} \left\{ \Delta^{d-1-i} \left[ \lambda\left( \dfrac{ (I^{n+1} + Q_i):x_{i+1}}{I^{n}+Q_i}\right)\right]
               -\lambda\left( \dfrac{(I^{n+1}+Q_{i}):x_{i+1}}{(QI^{n}+Q_{i}):x_{i+1}}\right) \right\}, & \mbox{ if } d \geq 2.
        \end{cases}
    \end{align*}
    Since $x_{1},\ldots,x_{d}$ is a superficial sequence of $I,$ from  \cite[Remark 2.2]{sam}, we have $w_{n}(Q,I)=0$ for sufficiently large $n.$ Furthermore, from \cite[Lemma 2.3]{sam}, if $d >1$, then 
        \begin{eqnarray}\label{wequation}
            w_{n}(Q,I)=w_{n}(Q',I')+\Delta^{d-1}\left[\lambda\left( \dfrac{(I^{n+1}:x_{1})}{I^{n}}\right)\right]-\lambda\left( \dfrac{I^{n+1}:x_{1}}{QI^{n}:x_{1}}\right),
         \end{eqnarray}
        for all $n \geq 1,$ where $'$ denotes image modulo $x_{1}.$
\end{definition}

The following proposition generalizes  \cite[Theorem~2.4]{sam}. 

\begin{proposition} \label{hilbcoeffextension}
    Let $(R,\m)$ be a Noetherian local ring of dimension $d \geq 1$ and $\depth R \geq d-1.$ Let $I$ be an $\m$-primary ideal  and let $Q=(x_{1},\ldots,x_{d})$ be a minimal reduction of $I$. 
     Then for all $n\geq 0$, 
    \begin{equation*}
    \Delta^{d}[P_{I}(n+1)-H_{I}(n+1)]=\lambda(l_{Q,I}(n))+w_{n}(Q,I)-\lambda(U_{I}(d;n)).
    \end{equation*}

\end{proposition}
\begin{proof}
 We prove by induction on the dimension $d.$ If  $d=1,$ then  $w_{n}(Q,I)=0$ and hence
   \begin{align*}
    \Delta[P_{I}(n+1)-H_{I}(n+1)]&=e_{0}(I)-\lambda(I^{n}/I^{n+1}) & \\
     &=\lambda(I^{n+1}/x_1 I^{n}) - \lambda((0:x_{1})\cap I^{n})
      \mbox{ (by Lemma \ref{formulahilbcoeffindimone})}. 
   \end{align*}
   
 Hence the statement is true for $d=1$.
Assume $d \geq 2$ and that our assertion is true  for $d-1.$  
From Remark \ref{remarksup}, we choose $x_1$ such that it is a non-zero divisor in $R$. Set $R'=R/(x_{1}),$ $I'=IR'$ and $Q'=QR'.$  We have
   \begin{align} \label{eqlength}  \nonumber
       \lambda(l_{Q,I}(n))
 &=\lambda\left(\frac{I^{n+1}+(x_{1})}{QI^{n}+(x_{1})}\right)
       +\lambda\left(\frac{QI^{n}+(x_{1})}{QI^{n}}\right)
       -\lambda\left(\frac{I^{n+1}+(x_{1})}{I^{n+1}}\right)\\ \nonumber
&=\lambda\left(\frac{{I}^{n+1}+(x_{1})}{{QI^{n}}+(x_{1})}\right)
       + \lambda\left(\frac{(x_{1})}{ x_1(QI^{n}:x_{1}) } \right)
       -  \lambda\left(\frac{(x_{1})}{ x_1(I^{n+1}:x_{1})}  \right)\\ 
       &= \lambda(l_{Q',I'}(n))
       + \lambda\left(\frac{(I^{n+1}:x_{1})}{ (QI^{n}:x_{1}) } \right).
   \end{align}
 Consider
  \begin{align*}  \nonumber
    \Delta^{d}[P_{I}(n+1)-H_{I}(n+1)]=& \Delta^{d-1}[\Delta[P_{I}(n+1)-H_{I}(n+1)]] & \\\nonumber
    &=\Delta^{d-1}[P_{I'}(n+1)-H_{I'}(n+1)+\lambda((I^{n+1}:x_{1})/I^{n})]  \text{ (by \cite[22.6]{Nagata})} & \\
    &=\Delta^{d-1}[P_{I'}(n+1)-H_{I'}(n+1)]+\Delta^{d-1}[\lambda((I^{n+1}:x_{1})/I^{n})] \\
    &= \lambda(l_{Q',I'}(n))+w_{n}(Q',I') -\lambda (U_{I'}(d-1;n))\\
   & \quad + \Delta^{d-1}[\lambda((I^{n+1}:x_{1})/I^{n})]  \text{ (by induction hypothesis)}\\ \nonumber
        &= \lambda(l_{Q,I}(n))+w_{n}(Q,I) 
        -\lambda (U_{I}(d;n)) \text{ (by Equations (\ref{wequation}) and (\ref{eqlength}))}. &&\qedhere
\end{align*} 
\end{proof}
As an immediate consequence,  we have the following corollary.

\begin{corollary}\label{e_i}
      Let $(R,\m)$ be a Noetherian local ring of dimension $d \geq 1$ and $\depth R \geq d-1.$ Let $I$ be an $\m$-primary ideal with minimal reduction $Q=(x_{1},\ldots,x_{d})$. Then  for $1\leq i\leq d$, we have 
    \[e_{i}(I)=\displaystyle \sum_{n=i-1}^{\infty}\binom{n}{i-1}\left[ \lambda(l_{Q,I}(n))+w_{n}(Q,I)-\lambda(U_{I}(d;n))\right].\]
\end{corollary}
\begin{proof}
  For $1\leq i \leq d$, from \cite[Proposition 2.9]{sam}, we have $e_{i}(I)=\displaystyle \sum_{n=i-1}^{\infty}\binom{n}{i-1} \Delta^{d}[P_{I}(n+1)-H_{I}(n+1)]$. Thus, the conclusion follows  directly from Proposition \ref{hilbcoeffextension}.
\end{proof}

The following lemma will be used in  our paper. 
\begin{lemma} \label{ozekinumericallemma}
\cite[Lemma 3.1]{ozeki}
 Let $l \geq 0$ be an integer. Suppose that $\{v_{n}\}_{n \geq 1}$ is the set of integers such that
 \begin{enumerate}[\normalfont(i)]
     \item $v_{n} \geq 0$ for all $n \geq 1,$
     \item $\displaystyle\sum_{n \geq 1}v_{n} \leq l$, and
     \item $v_{j}=0$ for all $j \geq n$ once $v_{n}=0$ for some $n \geq 1.$
 \end{enumerate}
Then we have $\displaystyle \sum_{n\geq 1}nv_{n} \leq \binom{l+1}{2}$. Further,  $\displaystyle \sum_{n\geq 1}nv_{n} = \binom{l+1}{2}$ if and only if $v_{n}=1$ for all $1 \leq n \leq l$, and $v_{n}=0$ for all $n \geq l+1.$  
\end{lemma}
   Note that for any $\m$-primary ideal $I$ with minimal reduction $Q,$ the  integers $\{\lambda(l_{Q,I}(n))\}_{n \geq 1}$ satisfies all conditions (i), (ii) and (iii) of Lemma \ref{ozekinumericallemma}.

\section{A Lower Bound for the Second Hilbert Coefficient} \label{section3}

In this section,  we establish a lower  bound for  the second Hilbert coefficient of an $\m$-primary ideal $I$.  
A lower bound for  $e_2(I)$ in a two dimensional Cohen-Macaulay local ring  was given by Rossi and Valla in  \cite[Theorem~3.1(a)]{rv}.  More precisely they proved that   $e_2(I) \geq e_1(I) - e_0(I) + \lambda( R/\widetilde{I})$.  We generalize this result to the case when the ring $R$ is almost Cohen-Macaulay. 
Before we prove our result for the lower bound, we give some preliminary lemmas. 

\begin{lemma} \label{2-II}
    Let $(R,\m)$ be a two-dimensional  local ring with $\depth R>0$ and let $I$ be an $\m$-primary ideal. Let $x_{1}, x_{2} $ be a superficial sequence of $I$ such that $Q=(x_{1},x_{2})$ is a minimal reduction of $I.$  Set $R' = R/(x_1)$ and $I' = IR'$. Then
    \begin{align*}
    e_{2}(I)= \displaystyle\sum_{n \geq 1} n\left[\lambda(l_{Q', I'}(n)
- \lambda(U_{I}(2;n))\right]
-  \displaystyle \sum_{n\geq 0}\lambda((I^{n+1}:x_1)/I^{n}).
\end{align*}
\end{lemma}

\begin{proof}
    From Proposition \ref{hibertcoeffsup} (iv),  we have 
 \begin{align*}
 \nonumber
e_{2}(I) 
& =  e_{2}(I')-  \displaystyle \sum_{n\geq 0}\lambda((I^{n+1}:x_1)/I^{n})    \\
&= \displaystyle\sum_{n \geq 1} n\left[\lambda(l_{Q', I'}(n)
- \lambda(U_{I}(2;n))\right]
-  \displaystyle \sum_{n\geq 0}\lambda((I^{n+1}:x_1)/I^{n}) \text{ (by Lemma \ref{formulahilbcoeffindimone})}. && \qedhere
\end{align*} 
\end{proof}

\begin{lemma} 
\label{pp}
    Let $(R,\m)$  local ring of dimension one  and let  $I$ be an $\m$-primary ideal with minimal reduction $Q=(x),$ where $x$ is a superficial element for $I$.  Then the following hold:
    \begin{enumerate}[\normalfont(i)]
        \item  
        \label{pp-1}
        For all  $n \geq e_{0}(I)-e_{1}(Q)$,  we have 
        $(0:x) \cap I^{n}=0$,
        \item 
        \label{pp-2}
        $\displaystyle \sum_{n\geq 0}\lambda((0:x) \cap I^{n})
        \leq (e_{0}(I)-e_{1}(Q))\lambda( H^{0}_{\m}(R)).$
    \end{enumerate} 
\end{lemma}
\begin{proof}

(\ref{pp-1}) We have the following natural injective map of $G(I)$-modules:
\begin{eqnarray}
  \label{ses bn(I)}                              0 
\longrightarrow \displaystyle \bigoplus_{n\geq 0}\frac{I^{n+1}+[(0:x) \cap I^{n}]}{I^{n+1}} 
\longrightarrow \displaystyle G(I).
\end{eqnarray}

As $x \in I$ is a superficial element of $I,$ thus from \cite[Theorem~1.2~(6)]{rv}, we have $(0:x)\cap I^{n}=0$ for sufficiently large $n$, therefore, the first term in the above sequence is Artinian. Hence  (\ref{ses bn(I)}) induces the injective map
 \begin{eqnarray} \label{pbks}
              0  
 \longrightarrow H^{0}_{G_{+}}\left( \bigoplus_{n\geq 0}\frac{I^{n+1} + [(0:x)\cap I^{n}]}{I^{n+1}}\right)  =\displaystyle\bigoplus_{n\geq 0}\frac{I^{n+1} + [(0:x)\cap I^{n}]}{I^{n+1}}
\longrightarrow H^{0}_{G_{+}}\left(G(I)\right),
         \end{eqnarray}
where $G_{+}=\displaystyle \bigoplus_{n \geq 1}I^{n}/I^{n+1}.$
Since  $H^{0}_{G_{+}}\left(G(I)\right)_{n}=0$ for all $n \geq a_{0}(G(I))+1$, where $a_{0}(G(I))=\sup\{n\in \mathbb Z: (H^{0}_{G+}(G(I)))_{n}\neq 0\}$,  thus, from  (\ref{pbks})
        we get that for all  $n \geq a_{0}(G(I))+1$, 
\begin{equation*}
    \frac{I^{n+1}+[(0:x)\cap I^{n}]}{I^{n+1}}\cong \frac{(0:x)\cap I^{n}}{(0:x)\cap I^{n+1}}=0  .
\end{equation*}
Thus, $(0:x) \cap I^{n}=0$ for all $n \geq a_{0}(G(I))+1$.
From \cite[Corollary~4.5]{mandal-nanduri}, 
we get $a_{0}(G(I))\leq e_{0}(I)-e_{1}(Q)-1.$ Thus, $(0:x) \cap I^{n}=0$ for all $n \geq e_{0}(I)-e_{1}(Q).$

(\ref{pp-2}) For every $n \geq 0,$ we have $\lambda((0:x)\cap I^{n}) \leq \lambda(0:x)$. Hence, from (\ref{pp-1}), we have
\begin{equation*}
    \sum_{n\geq0}\lambda((0:x) \cap I^{n})
    \leq \sum_{n=0}^{e_{0}(I)-e_{1}(Q)-1}\lambda(0:x)=(e_{0}(I)-e_{1}(Q))\lambda(0:x)\leq (e_{0}(I)-e_{1}(Q))\lambda\left(H^{0}_{\m}(R)\right). 
\end{equation*}
The last inequality follows from \cite[Lemma 2.3]{rv}.
\end{proof}

We  now establish a lower bound for $e_2(I)$. If the ring is  Cohen-Macaulay, then we recover Marley's result  \cite[Corollary~2(3)]{marley}.

 \begin{proposition} \label{newandfinalproofoflowerbound}
      Let $(R,\m)$ be a  local ring of dimension $d\geq 2$ and let $I$ be an $\m$-primary ideal.  If $\depth G(I) \geq d-1,$ then  we have the following:
 \begin{enumerate} [\normalfont (i)]
     \item 
     \label{newandfinalproofoflowerbound-1}
     For any minimal reduction $Q$ of $I$, 
      \begin{equation*} 
      e_{2}(I) \geq  \mathrm{g}_{s}(I)-\mathrm{g}_{s}(Q)+e_{1}(Q)\dbinom{e_{0}(I)-e_{1}(Q)}{2}.
       \end{equation*}
              \item 
       \label{newandfinalproofoflowerbound-2}
      \cite[Corollary~2(3)]{marley}
       If $R$ is Cohen-Macaulay, then 
       $e_{2}(I)\geq e_{1}(I)-e_{0}(I)+\lambda(R/I)$.    
 \end{enumerate}
 \end{proposition}
 \begin{proof}
 (\ref{newandfinalproofoflowerbound-1}) \normalsize
   Let  $Q=(x_1,\ldots,x_{d})$.  Since  $\depth G(I)
     \geq d-1,$ thus,    $\depth R
     \geq d-1,$ therefore, we can choose   $x_1,\ldots,x_{d-1}$ to be a regular sequence   for $R.$ We set $R_{d-1}=R/(x_1,\ldots x_{d-1}),$ $I_{d-1}=IR_{d-1}$ and $Q_{d-1}=QR_{d-1}.$ From Proposition \ref{hibertcoeffsup}, we have $e_{2}(I)=e_{2}(I_{d-1}).$ Hence
     \begin{align} \label{depthd-1upperbound} \nonumber 
    e_{2}(I)
    &= \displaystyle\sum_{n \geq 1} n\left[\lambda(l_{Q_{d-1},I_{d-1}}(n))
- \lambda(U_{I}(d;n))\right] \text{ (by Lemma \ref{2-II} $\&$ $\depth G(I) \geq d-1$)}\\ \nonumber
  &= \sum_{n \geq 1} n\left[\lambda(l_{Q_{d-1},I_{d-1}}(n))- \lambda(U_{I}(d;n))\right]  + \sum_{n \geq 0}\lambda(l_{Q_{d-1},I_{d-1}}(n))
 -  \sum_{n \geq 0}\lambda(l_{Q_{d-1},I_{d-1}}(n))\\ \nonumber
&= \sum_{n \geq 1} n\left[\lambda(l_{Q_{d-1},I_{d-1}}(n))- \lambda(U_{I}(d;n))\right]\\ \nonumber
&  \quad +\left(e_{1}(I_{d-1})+\displaystyle\sum_{n \geq 0}\lambda(U_{I}(d;n))\right)
-  \sum_{n \geq 0}\lambda(l_{Q_{d-1},I_{d-1}}(n))   \text{ (by Lemma \ref{formulahilbcoeffindimone})} \\ \nonumber
&=  \sum_{n \geq 1} (n-1) \left[\lambda(l_{Q_{d-1},I_{d-1}}(n))
- \lambda(U_{I}(d;n))\right] &\\
&+e_1 (I_{d-1}) 
- \lambda\left(  \frac{I_{d-1}}{Q_{d-1}}\right) 
+ \lambda(U_{I}(d;0))  \\ \nonumber
& \geq e_1(I) - \lambda(I/Q) 
- \sum_{n= 1}^{e_{0}(I)-e_{1}(Q)-1} (n-1)\lambda(U_{I}(d;n)) 
 \text{ (Lemma~\ref{pp} $\&$ Proposition~\ref{hibertcoeffsup})}\\ \nonumber
& =e_1(I) - e_0 (I) + \lambda(R/I) 
-e_{1}(Q)+e_0(Q) - \lambda(R/Q)  - \sum_{n = 1}^{e_{0}(I)-e_{1}(Q)-1}  (n-1)\lambda(U_{I}(d;n))+e_{1}(Q)\\  \nonumber
&= \mathrm{g}_{s}(I)-\mathrm{g}_{s}(Q)-
 \sum_{n = 1}^{e_{0}(I)-e_{1}(Q)-1} (n-1)\lambda(U_{I}(d;n))+e_{1}(Q).
\end{align}

As $\dim R_{d-1}=1$, applying \cite[Lemma 2.3]{rv} to $R_{d-1},$ we get 
\begin{eqnarray} \label{ud}
    \lambda (U_{I}(d;n)) 
    \leq \lambda(H_{\mathfrak \m_{d-1}}^{0}(R_{d-1}))=-e_{1}(Q_{d-1})=-e_{1}(Q) \text{ (from \cite[Proposition 3.1]{mousumi})}. 
 \end{eqnarray}
Therefore, by applying Lemma \ref{pp} on $R_{d-1}$ and and substituting  in Equation (\ref{depthd-1upperbound}), we get 
    $e_{2}(I) \geq \mathrm{g}_{s}(I)-\mathrm{g}_{s}(Q)+e_{1}(Q)\dbinom{e_{0}(I)-e_{1}(Q)}{2}.$

(\ref{newandfinalproofoflowerbound-2})
 If $R$ is Cohen-Macaulay, then  $\mathrm{g}_{s}(Q)=e_{1}(Q)=0$ and hence  $\mathrm{g}_{s}(I)-\mathrm{g}_{s}(Q)+e_{1}(Q)\dbinom{e_{0}(I)-e_{1}(Q)}{2} = e_{1}(I)-e_{0}(I)+\lambda(R/I)$. Substituting in  (\ref{newandfinalproofoflowerbound-1}),  we get the result. 
 \end{proof}

\section{Upper Bounds for the Second Hilbert Coefficient} \label{section4}
In this section, we establish an  upper bound for $e_2(I)$ for an $\m$-primary ideal $I$. 
We begin with an alternate proof of  McCune's result which gives an upper bound for $e_2(Q)$, where $Q$ is a parameter ideal in a local ring of depth at least $d-1$. Recall that the \textit{postulation number} of an ideal $I$ is defined as:
$$n(I)=\min\{n \in \mathbb Z: P_{I}(t)=H_{I}(t) \text{ for } t>n\}.$$

\begin{proposition}
\label{lori}
\cite[Theorem 3.5]{lori}  Let $(R, \m)$ be a Noetherian local ring of dimension $d \geq 2$ and $\depth R \geq d-1.$ Let $Q$ be a parameter ideal. Then the following hold:
\begin{enumerate}[\normalfont(i)]
    \item
    \label{lori-1}
    $e_{2}(Q) \leq 0.$
    \item 
    \label{lori-2}
    $e_{2}(Q) = 0$ if and only if $n(Q) <2-d$ and  $\depth G(Q) \geq d-1.$
    \item 
    \label{lori-3} If $e_{2}(Q)=0$ then $e_{3}(Q)=\cdots=e_{d}(Q)=0.$
\end{enumerate}
\end{proposition} 
\begin{proof}
Let $Q=(x_{1},\ldots,x_{d}),$ where $x_{1},\ldots,x_{d}$ is a system of parameters in $R.$ By Remark \ref{remarksup}, we can choose $x_{1} $  a superficial element of $Q$ which is also a non-zero divisor. Set $R'=R/(x_{1})$ and $Q'=QR'.$

(\ref{lori-1})
We prove by induction on dimension $d$. Let   $d=2$. Since $Q'$ is a parameter ideal in $R',$  $\lambda(l_{Q',Q'}(n))=0$ for all $n\geq 1$ and hence from  Lemma~\ref{2-II}, we get
\begin{align}
 \label{e2-para} 
       e_{2}(Q)
&= - \displaystyle\sum_{n \geq 1} n\lambda(U_{Q}(2;n))
-  \displaystyle \sum_{n\geq 0}\lambda((Q^{n+1}:x_{1})/Q^{n})
\leq
 0 
    \end{align}
If $d\geq 3$, then from Proposition \ref{hibertcoeffsup} and induction hypothesis  we get $e_{2}(Q)=e_{2}(Q') \leq 0.$ 

(\ref{lori-2})
Let $d=2$. If   equality holds in  (\ref{e2-para}), then $(Q^{n+1}:x_1) = Q^n$  for all $n \geq 0$ and hence 
$x_{1}^{*}$ is a non zero-divisor of $G(Q)$   which implies that  $\depth G(Q)\geq 1.$   Moreover, $e_0(Q) = e_0(Q^{\prime})$ and $\lambda 
(U_{Q} (2;n)) = 0$ for all $n \geq 1$. Hence, from Proposition \ref{hibertcoeffsup} (v) and \cite[2.4, page 21]
{rv},  the Hilbert Series of $G(I)$ is 
  \begin{eqnarray*}
   HS_{Q}(t) 
   = \frac{HS_{Q'}(t)} {1-t} =\frac{e_0(Q) + U_{Q}(2;0) - U_{Q}(2;0) t}{(1-t)^2}
     \end{eqnarray*}
 Therefore, the Hilbert-Samuel polynomial is
 \begin{eqnarray*}
  P_{Q}(n)=e_{0}(Q)\binom{n+1}{2}+U_{Q}(2;0)n.
 \end{eqnarray*}
Thus, $P(n) = H(n)$ for all $n \geq 0$ which implies that  $n(Q)<0.$ 
  Conversely, suppose $n(Q)<0$, then $e_{2}(Q)= P(0) = H(0)=0 $.

Now assume that  $d \geq 3$ and that our assertion holds for $d-1.$ If  $e_{2}(Q)=0,$ then $e_{2}(Q')=e_2(Q) =0$. 
 Hence, by  induction hypothesis, $n(Q') < 2-(d-1)$ and  $\depth G(Q') \geq d-2$. 
From \cite[Lemma 2.8]{marley}, we have $n(Q)+1=n(Q'),$ therefore, $n(Q)<2-d,$  and from Sally's descent we get $\depth G(Q) \geq d-1$. 

 Conversely, suppose that $n(Q)<2-d$ and $\depth G(Q) \geq d-1.$ Then from \cite[Lemma 2.8]{marley} $n(Q')=n(Q)+1<2-(d-1)$. From Sally's descent, we have  and $\depth G(Q')\geq d-2$. Hence by induction hypothesis, $e_{2}(Q')=0.$ From Proposition \ref{hibertcoeffsup}, $e_{2}(Q)=e_{2}(Q')=0.$
 
(\ref{lori-3}) If $d=2$, the the result trivially holds true. Let $d \geq 3$. Then by induction hypothesis 
 $e_{3}(Q')=\cdots=e_{d-1}(Q')=0$. From Proposition~\ref{hibertcoeffsup} (ii) $e_{3}(Q)=\cdots=e_{d-2}(Q)=0$ and from  Proposition~\ref{hibertcoeffsup} (iii) $e_{d-1}(Q)=0$.   
It remains to show that $e_d(Q)=0$. 
  Since $\depth G(Q) \geq d-1,$  from the proof of \cite[Corollary 2.6]{sam}, we get $w_{n}(Q,Q)=0$ for all $n \geq 1.$ Note that since $e_{2}(Q')=0$  from (\ref{e2-para}) we get that  $U_{Q'}(d-1;n)=0$ for all $n\geq 1$. Hence, $U_{Q}(d;n)=0$ for all $n\geq 1.$  Therefore, from Corollary \ref{e_i}, $e_{d}(Q) =0.$ 
\end{proof} 

From \cite[Example 3.7]{lori}, it is evident that the assumption $\depth R \geq d-1$ is essential for the non-positivity of $e_2(Q)$. 

We now give an  improvement of McCune's result. In fact we obtain upper bound for $e_2(Q)$ under the assumption that  $\depth R \geq d-2,$ where $Q$ is a parameter ideal. 
\begin{proposition} 
\label{parameterbounddepthzero}
     Let $(R,\m)$ be a Noetherian local ring of dimension $d \geq 2$ with $\depth R \geq d-2$ and    let $Q$ be a parameter ideal in $R$. Then    $e_2(Q) \leq \lambda\left(H_{\m_{d-2}}^{0}(R_{d-2})\right)$, where $\underline{x}=x_{1},\dots,x_{d-2}$ is a regular sequence in $R,$ and $R_{d-2}=R/(\underline{x})$ and $\m_{d-2}=\m R_{d-2}.$   
\end{proposition}
\begin{proof} We proceed by induction on the dimension $d.$ Suppose $d=2.$
If $\depth R = 1$, then $H_{\m}^{0}(R)=0$ and the result follows from Proposition~\ref{lori} (\ref{lori-1}). If $\depth R = 0$, we set $\overline{R}=R/\left(H_{\m}^{0}(R)\right)$ and $\overline{Q}=Q\overline{R}$. By \cite[Proposition 2.3]{rv}
 $ e_{2}(Q)
        =e_{2}(\overline{Q})+\lambda\left(H_{\m}^{0}(R)\right) $. 
   Since    $\depth \overline{R} \geq 1$, by  Proposition~\ref{lori}(\ref{lori-1}), $e_{2}(\overline{Q})\leq0$ which implies that  $ e_{2}(Q) \leq \lambda\left(H_{\m}^{0}(R)\right)$. 

 Suppose $d \geq 3 $ and $\depth R \geq d-2.$ Assume that our assertion is true for $d-1.$ Put $Q_{d-2} = Q R_{d-2}$.  Since $\underline{x}$ is a regular sequence in $R,$ from Proposition \ref{hibertcoeffsup} (ii), we have $e_{2}(Q)=e_{2}(Q_{d-2}).$  Now from the induction hypothesis, we have $e_{2}(Q_{d-2}) \leq \lambda \left(H_{\m_{d-2}}^{0}(R_{d-2})\right).$ Therefore,  $e_{2}(Q) \leq \lambda \left(H_{\m_{d-2}}^{0}(R_{d-2})\right).$
\end{proof}

In a Buchsbaum local ring, for any parameter ideal $Q$, from \cite[Corollary 4.2]{trung3}, the description of all the Hilbert coefficients is given explicitly as follows: 
$$e_{i}(Q)=(-1)^i\displaystyle\sum_{j=0}^{d-i}\binom{d-i-1}{j-1}\lambda\left(H_{\mathfrak m}^{j}(R)\right) \text{ for all $i=1,\ldots,d$},$$
where $\binom{d-i-1}{-1}=0$ if $i \neq d$ and $\binom{-1}{-1}=1.$
In particular, if $R$ is a two-dimensional Buchsbaum local ring with $\depth R = 0$, then for every parameter ideal $Q$ one has $e_{2}(Q)=\lambda\left(H_{\mathfrak m}^{0}(R)\right)$. Thus, in the two-dimensional Buchsbaum case, $e_{2}(Q)$ attains the bound  in Proposition \ref{parameterbounddepthzero}. However, in the following example, we show that the upper bound  on $e_{2}(Q)$ obtained in Proposition \ref{parameterbounddepthzero} is optimal when the ring is not Buchsbaum.
\begin{example}
\label{ex vanishing e_2}
    Let  $S =\mathbb Q[[X,Y,Z]]$ where $X,Y,Z$ are variables,   $\mathfrak n = (X,Y,Z)$ and $K =
    (X^2) \cap (X^3,Y^3,Z^4)$. 
     Put  $R=S/K= \mathbb Q[[x,y,z]]$ and $\m=(x,y,z)$,  where $x,y,z$ denote the images of $X,Y,Z$ in $R.$  Note that $\dim R=2$  and  $\depth R=0.$  Since  $(K : \n^n) = (X^2)$, for all $n\geq 6,$  
    $$H_{\m}^{0}(R)\cong (x^2)
    = \dfrac{(X^2)}{ (X^3,X^2Y^3,X^2Z^4) }  
    \cong \dfrac{S}{ (X,Y^3,Z^4)},$$
    which gives $\lambda\left(H_{\m}^{0}(R)\right)=12.$

    Let $Q=(y,z)$  be a parameter ideal of $R.$  Then 
 \begin{align*}
\lambda \left( \frac{R}{Q^n}\right)
&= \lambda \left( \frac{S}{ K + Q^n}\right) & \\
&= \lambda \left( \frac{S}{  (X) + K + Q^n}\right) + \lambda \left( \frac{S}{   (K + Q^n) : (X)}\right) & \\
&= \lambda \left( \frac{S}{  (X) + (Y,Z)^n}\right) + \lambda \left( \frac{S}{   ( X^2, XY^3, XZ^4) + Q^n}\right) & \\
&= 2 \lambda \left( \frac{S}{  (X) + (Y,Z)^n}\right) + \lambda \left( \frac{S}{   ( X, Y^3, Z^4 ) + Q^n}\right) & \\
&= 2\binom{n+1}{2}+12  \mbox{ for all } n \geq 7.
 \end{align*}  
 Hence,
       $e_{2}(Q) = 12= \lambda\left(H_{\m}^{0}(R)\right)$, but $R$ is not Buchsbaum since $\m H^{0}_{\m}(R) \neq 0$.  
\end{example}

Before we establish the upper bound for the second Hilbert coefficient of any $\m$-primary ideal $I$, we prove some preliminary results
 on the sectional genera $\mathrm{g}_s(I).$ For a parameter ideal $Q$, in 
 \cite[Lemma 3.2]{go}  Goto and Ozeki discussed the behavior of $\mathrm{g}_s(Q)$  going modulo a superficial element. We extend this result to any  $\m$-primary ideal $I$. 
Though the same proof goes through, we prove it as it is used in our main result.

\begin{lemma} \label{lemmaa1}
     Let $(R,\m)$ be a Noetherian local ring of dimension $d \geq 2$ and let $I$ be an $\m$-primary ideal with minimal reduction $Q$. Let $x \in Q$ be a superficial element of both $Q$ and $I.$ Then  
\[\mathrm{g}_{s}(I)= \begin{cases}
    \mathrm{g}_{s}(I/(x))+\lambda(0:_{R}x), \text{ if } d=2,\\
    \mathrm{g}_{s}(I/(x)), \text{ if } d\geq3.
\end{cases}\]
\end{lemma}

\begin{proof}
   Set $R'=R/(x)$, $I'=IR'$ and $Q'=QR'.$  Let  $d=2$. Then 
  \begin{align*}
     \mathrm{g}_{s}(I/(x)) 
&=\lambda(R'/I')-e_{0}(I')+e_{1}(I')&\\
 &= \lambda(R/I)-e_{0}(I)+e_{1}(I)-\lambda(0:_{R}x)   \mbox{ (by Proposition \ref{hibertcoeffsup})} \\
 &=\mathrm{g}_{s}(I)-\lambda(0:_{R}x).
  \end{align*}
    For $d \geq 3,$ from Proposition \ref{hibertcoeffsup},  we have $e_{0}(I)=e_{0}(I')$ and  $e_{1}(I)=e_{1}(I')$. Therefore, $\mathrm{g}_{s}(I')=\mathrm{g}_{s}(I).$
\end{proof}

\begin{lemma} 
 \label{2-prime}
Let $(R,\m)$ be a  Noetherian local ring of dimension $d$ with $\depth R \geq d-1$ and let  $I$ be an $\m$-primary ideal. Let $x_{1}, \dots ,x_{d} $ be a superficial sequence of $I$ such that $Q=(x_{1},\ldots,x_{d})$ is a minimal reduction of $I$. Assume that $x_1,\ldots,x_{d-1}$ is a regular sequence in $R$ and    set $R_{d-1}=R/(x_{1},\ldots,x_{d-1})$, $I_{d-1}=IR_{d-1}$ and $Q_{d-1}=QR_{d-1}.$ Then the following hold:
\begin{enumerate} [\normalfont(i)]
    \item  \label{2-prime-1} $\displaystyle \sum_{n \geq 1}\lambda( l_{Q_{d-1},I_{d-1}}(n))
  \leq \mathrm{g}_{s}(I)-\mathrm{g}_{s}(Q)-(e_{0}(I)-e_{1}(Q))e_{1}(Q)$. Further, equality holds if and only if $R$ is Cohen-Macaulay. 
  \item \label{2-prime-2} If $\depth G(I) \geq d-1,$ then $r_{Q}(I) \leq \mathrm{g}_{s}(I)-\mathrm{g}_{s}(Q)-(e_{0}(I)-e_{1}(Q))e_{1}(Q)+1.$
\end{enumerate}
 \end{lemma}
\begin{proof}  
(\ref{2-prime-1}) Note that $l_{Q_{d-1},Q_{d-1}}(n) = 0$ for all $n \geq 0$ and $e_0(I_{d-1}) = e_0(Q_{d-1})$. Hence  from Lemma \ref{formulahilbcoeffindimone}, we have $e_{1}(I_{d-1})=\displaystyle\sum_{n \geq 0}[\lambda(l_{Q_{d-1},I_{d-1}}(n))-\lambda(U_{I}(d;n))]$. Therefore
\begin{align} \nonumber \label{cmequality6} \nonumber
    \sum_{n \geq 1} \lambda (l_{Q_{d-1},I_{d-1}}(n))
&=\sum_{n \geq 1} \lambda (l_{Q_{d-1},I_{d-1}}(n)) 
- \sum_{n \geq 0} \lambda (l_{Q_{d-1},Q_{d-1}}(n)) -e_{0}(I_{d-1})+e_{0}(Q_{d-1})  \\ \nonumber
&=\sum_{n \geq 0} \lambda (l_{Q_{d-1},I_{d-1}}(n)) - \lambda(I_{d-1}/Q_{d-1}) 
- \sum_{n \geq 0} \lambda (l_{Q_{d-1},Q_{d-1}}(n)) -e_{0}(I_{d-1})+e_{0}(Q_{d-1}) \\ \nonumber
  &=e_{1}(I_{d-1})+\sum_{n\geq 0}\lambda(U_{I}(d;n))-\lambda(I/Q)-e_{1}(Q_{d-1})-\sum_{n\geq 0}\lambda(U_{Q}(d;n))\\ \nonumber
  & \quad -e_{0}(I_{d-1})+e_{0}(Q_{d-1})  \\ 
 &\leq e_{1}(I_{d-1})-e_{0}(I_{d-1})+\lambda(R/I)-e_{1}(Q_{d-1})+e_{0}(Q_{d-1}) -\lambda(R/Q))+\sum_{n\geq 0}\lambda(U_{I}(d;n))  \\ \nonumber
    &= \mathrm{g}_{s}(I_{d-1})-\mathrm{g}_{s}(Q_{d-1})+\sum_{n\geq 0}\lambda(U_{I}(d;n))&\\ \nonumber
    &= \mathrm{g}_{s}(I)-\mathrm{g}_{s}(Q)+\sum_{n\geq 0}\lambda(U_{I}(d;n))   \text{ (by Lemma \ref{lemmaa1})}\\ \nonumber
    &=\mathrm{g}_{s}(I)-\mathrm{g}_{s}(Q)+(e_{0}(I_{d-1})
    -e_{1}(Q_{d-1}))\lambda(H^{0}_{\m_{d-1}}(R_{d-1})) \\ \nonumber 
    & \quad \text{ (applying Lemma \ref{pp} on $R_{d-1}$)}\\ \nonumber
    & =\mathrm{g}_{s}(I)-\mathrm{g}_{s}(Q)-(e_{0}(_{d-1})-e_{1}(Q_{d-1}))e_{1}(Q_{d-1}) 
    \mbox { (by \cite[Proposition 3.1]{mousumi})}\\  \nonumber
     & = \mathrm{g}_{s}(I)-\mathrm{g}_{s}(Q)-(e_{0}(I)-e_{1}(Q))e_{1}(Q)  
     \mbox{\text{ (by Proposition \ref{hibertcoeffsup})}}.
\end{align}

\normalsize
Suppose the equality holds, then from (\ref{cmequality6}) we get $\displaystyle \sum_{n \geq 0}\lambda(U_{Q}(d;n))=0,$ which implies $U_{Q}(d;n)=0$ for all $n \geq 0.$  In particular, for $n=0,$ we get $(0:_{R_{d-1}}x_{d})=0$, this implies $x_{d}$ is a non zero-divisor of $R_{d-1}$. Thus, $R$ is Cohen-Macaulay. The converse follows  from the fact that $x_{d}$ is a non zero-divisor of $R_{d-1}$ then $(0:_{R_{d-1}}x_{2})=0$, this implies   $U_{Q}(d;n)=0$ for all $n \geq 0$ and equality holds.

(\ref{2-prime-2}) Since $\depth G(I) \geq d-1,$ thus for all $n \geq 0$, $I^{n+1}\cap (\underline{x})=(\underline{x})I^{n}$, where $\underline{x}^{*}=x_{1}^{*},\ldots,x_{d-1}^{*}$ is a regular sequence in $G(I),$ and hence,   we have  
\begin{equation} \label{iso-length1}
        \frac{I^{n+1}+(\underline{x})}{QI^{n}+(\underline{x})} 
 \cong \frac{I^{n+1}}{I^{n+1} \cap [QI^{n}+(\underline{x})]} 
 \cong \frac{I^{n+1}}{QI^{n}+[I^{n+1}\cap(\underline{x})]}
 \cong l_{Q,I}(n).
\end{equation}
Therefore, from (\ref{2-prime-1}), we get $$\displaystyle \sum_{n \geq 1}\lambda( l_{Q,I}(n))=\displaystyle \sum_{n \geq 1}\lambda( l_{Q_{d-1},I_{d-1}}(n))\leq  \mathrm{g}_{s}(I_{d-1})-\mathrm{g}_{s}(Q_{d-1})-(e_{0}(I_{d-1})-e_{1}(Q_{d-1}))e_{1}(Q_{d-1}).$$
From Proposition \ref{hibertcoeffsup},  $e_{0}(I) = e_{0}(I_{d-1})$ and $e_{1}(Q) = e_{1}(Q_{d-1})$. By Lemma \ref{lemmaa1}, $\mathrm{g}_s(I) = \mathrm{g}_s(I_{d-1})$ and $\mathrm{g}_s(Q) = \mathrm{g}_s(Q_{d-1})$. Thus, $\displaystyle \sum_{n \geq 1}\lambda( l_{Q,I}(n))\leq \mathrm{g}_{s}(I)-\mathrm{g}_{s}(Q)-(e_{0}(I)-e_{1}(Q))e_{1}(Q).$ Further, $r_{Q}(I)-1\leq\displaystyle \sum_{n \geq 1}\lambda( l_{Q,I}(n)),$ therefore, $r_{Q}(I)\leq \mathrm{g}_{s}(I)-\mathrm{g}_{s}(Q)-(e_{0}(I)-e_{1}(Q))e_{1}(Q)+1.$
\end{proof}

We now prove the main result of this section. We establish an upper bound for $e_{2}(I)$ in terms of sectional genera  $\mathrm{g}_{s}(I)$ and $\mathrm{g}_{s}(Q)$ of ideals $I$ and $Q$ respectively. We also analyze the case when equality holds, highlighting its implications for the depth of the associated graded ring and the reduction number of $I.$  This extends Ozeki's result \cite[Theorem 3.4]{ozeki} to  non Cohen-Macaulay local rings.

\begin{theorem} \label{finalupperbounde2}
    Let $(R,\m)$ be a Noetherian local ring of dimension $d\geq 2$ and  $\depth R \geq d-1.$  Then  for any $\m$-primary ideal  $I$   and a  minimal reduction $Q=(x_1,\ldots,x_{d})$ of $I,$ 
     \begin{equation} \label{first bound e_2 equalityyy}
         e_2(I) \leq \displaystyle\binom{\mathrm{g}_s(I) - \mathrm{g}_s(Q) - (e_{0}(I)-e_{1}(Q))e_1(Q) + 1}{2}-\displaystyle \sum_{n=1}^{e_{0}(I)-e_{1}(Q)-1}
           n\lambda(U_{I}(d;n)).
     \end{equation}
 Further, suppose equality holds in  (\ref{first bound e_2 equalityyy}), then the following are true: 
 \begin{enumerate}[\normalfont(i)]
     \item \label{depthh1}   $\depth G(I) \geq d-1.$
     \item  \label{depthh2}  If $I$ is not a parameter ideal  then 
          $r_{Q}(I)=\mathrm{g}_s(I) - \mathrm{g}_s(Q) - (e_{0}(I)-e_{1}(Q))e_1(Q)+1.$
          \item \label{depthh3} Set $t=\mathrm{g}_{s}(I)-\mathrm{g}_s(Q) - (e_{0}(I)-e_{1}(Q))e_1(Q).$ For $3 \leq i \leq d,$ we have 
   \[\displaystyle\binom{t + 1}{i}+e_{1}(Q)\binom{e_{0}(I)-e_{1}(Q)}{i} \leq e_{i}(I) \leq\displaystyle\binom{t + 1}{i}.\]
 \end{enumerate}
  
\end{theorem}

\begin{proof} 
 We prove by induction on dimension $d$. For $d=2,$ since $\depth R>0,$ we can choose $x_{1}$ a non zero-divisor for $R.$ Set $R'=R/(x_{1}),$ $I'=IR'$ and $Q'=QR'$. Hence
\small{\begin{align} \nonumber \label{newandfinalproofofbound}
    e_{2}(I) &= \displaystyle\sum_{n \geq 1} n\left[\lambda(l_{Q', I'}(n)
- \lambda(U_{I}(2;n))\right]
-  \displaystyle \sum_{n\geq 0}\lambda((I^{n+1}:x_{1})/I^{n}) \text{ (by Lemma \ref{2-II})}\\ \nonumber
& \leq \displaystyle\sum_{n \geq 1} n\lambda(l_{Q', I'}(n))
-\displaystyle \sum_{n= 1}^{e_{0}(I)-e_{1}(Q)-1}n\lambda(U_{I}(2;n)
-\displaystyle \sum_{n\geq 0}\lambda((I^{n+1}:x_{1})/I^{n}) \text{ (by Proposition \ref{pp}(\ref{pp-1}))}\\
& \leq \displaystyle\binom{\mathrm{g}_s(I) - \mathrm{g}_s(Q) - (e_{0}(I)-e_{1}(Q))e_1(Q) + 1}{2}-\displaystyle \sum_{n=1}^{e_{0}(I)-e_{1}(Q)-1}
           n\lambda(U_{I}(2;n)) \text{  (by Lemma \ref{2-prime} and Lemma \ref{ozekinumericallemma})}
\end{align}}

  \normalsize
  Suppose equality holds in (\ref{newandfinalproofofbound}). Then   $\lambda\left((I^{n+1}:x_{1})/I^{n}\right)=0$ for all $n \geq 0$ which implies $\depth G(I) \geq 1.$

Assume $d \geq 3$ and that our assertion holds for $d-1.$   
From Proposition \ref{hibertcoeffsup}, we have $e_{0}(I) = e_{0}(I')$, $e_{1}(Q) = e_{1}(Q')$ and $e_{2}(I) = e_{2}(I')$. By Lemma \ref{lemmaa1}, $\mathrm{g}_s(I) = \mathrm{g}_s(I')$ and $\mathrm{g}_s(Q) = \mathrm{g}_s(Q')$. 
Since $e_{2}(I) = e_{2}(I')$ and $U_{I}(d;n)=U_{I'}(d-1;n),$ 
we get (\ref{first bound e_2 equalityyy}). Now suppose the upper bound is attained for $e_2(I)$, then it is also attained for $e_2(I').$ Therefore, by induction hypothesis, we get $\depth G(I') \geq (d-1)-1=d-2,$ and  thus, by Sally's descent $\depth G(I) \geq d-1.$ 

(\ref{depthh2})
Let $d \geq 2$. Assume that  $I$ is not a parameter ideal and equality holds in  (\ref{first bound e_2 equalityyy}), this implies $\depth G(I) \geq d-1.$ Thus, from  (\ref{newandfinalproofofbound}), we get $\displaystyle \sum_{n \geq 1}\lambda( l_{Q_{d-1},I_{d-1}}(n))=\binom{ \mathrm{g}_{s}(I)-\mathrm{g}_{s}(Q)-(e_{0}(I)-e_{1}(Q))e_{1}(Q)+1}{2}.$  Also, since $\depth G(I) \geq d-1,$ therefore, from Lemma \ref{ozekinumericallemma} and (\ref{iso-length1}), we get
\begin{eqnarray}
\label{length of lnq}
 \lambda( l_{Q,I}(n))   
 = \lambda( l_{Q_{d-1},I_{d-1}}(n))
 = \begin{cases}
 1 & n \leq \mathrm{g}_{s}(I)-\mathrm{g}_{s}(Q)-(e_{0}(I)-e_{1}(Q))e_{1}(Q)\\
 0 & n \geq \mathrm{g}_{s}(I)-\mathrm{g}_{s}(Q)-(e_{0}(I)-e_{1}(Q))e_{1}(Q)+1.
 \end{cases}
\end{eqnarray}
Therefore, $r_{Q}(I)= \mathrm{g}_{s}(I)-\mathrm{g}_{s}(Q)-(e_{0}(I)-e_{1}(Q))e_{1}(Q)+1.$

(\ref{depthh3})
Set $t=\mathrm{g}_{s}(I)-\mathrm{g}_s(Q) - (e_{0}(I)-e_{1}(Q))e_1(Q)$  and $c=e_{0}(I)-e_{1}(Q)-1.$ Suppose equality holds in (\ref{first bound e_2 equalityyy}), then from   (\ref{length of lnq}), we get  $\lambda( l_{Q,I}(n) )=1$ for all $1 \leq n \leq t,$ and $\lambda( l_{Q,I}(n) )=0$ for all $n \geq t+1.$  From Corollary \ref{e_i}, we have $e_{i}(I)\leq \displaystyle\sum_{n= i-1}^{t}\binom{n}{i-1}=\binom{t+1}{i}$,  since $\depth G(I) \geq d-1,$ thus $w_{n}(Q,I)=0$ for all $n.$ 
Again from Corollary \ref{e_i} and (\ref{ud}), we get that for all $i=3, \ldots, d$
\begin{eqnarray*}
    e_{i}(I)
\geq  \sum_{n=i-1}^{t} \binom{n}{i-1} + \sum_{n=i-1}^c \binom{n }{i-1}e_1(Q)=\binom{t+1}{i}+\binom{c+1}{i}e_1(Q).
\end{eqnarray*}
As $e_1(Q)\leq 0$ we get the required bounds.
  \end{proof}

As a  corollary, we recover Ozeki's bound \cite[Proposition 3.2]{ozeki} for  $e_{2}(I)$ in the case when the ring   is Cohen-Macaulay. 

  \begin{corollary}\cite[Proposition 3.2, Theorem 3.4]{ozeki}\label{cml}
     Let $(R,\m)$ be a Cohen-Macaulay local ring of dimension $d\geq 2$. Let $I$ be an $\m$-primary ideal with minimal reduction  $Q$. 
    Then
    \begin{equation} \label{ozekibounde2}
        e_2(I)
\leq \displaystyle\binom{\mathrm{g}_s(I)  + 1}{2}.
    \end{equation}
    Suppose equality holds in  (\ref{ozekibounde2}). Then the following are true:
        \begin{enumerate} [\normalfont(i)]
      \item  $\depth G(I) \geq d-1.$
\item  If $I$ is not a parameter ideal  then 
          $r_{Q}(I)=\mathrm{g}_s(I) +1.$
\item 
For $3 \leq i \leq d,$ we have $ e_{i}(I) =\displaystyle\binom{\mathrm{g}_s(I)  + 1}{i}.$
    \end{enumerate}
\end{corollary}

\begin{proof}
    Since $R$ is a Cohen-Macaulay local ring,  and $Q$ is a parameter ideal, thus
$\mathrm{g}_{s}(Q)=e_{1}(Q)=0$. Also,
 $\lambda (U_{I}(d;n)) = 0$ for all $n \geq 0$.  Therefore, from Theorem \ref{finalupperbounde2}, we get $e_{2}(I) \leq \dbinom{\mathrm{g}_{s}(I)+1}{2}.$
    
    Suppose $e_{2}(I) = \dbinom{\mathrm{g}_{s}(I)+1}{2}$, then from  Theorem \ref{finalupperbounde2} (\ref{depthh1}), we have $\depth G(I) \geq d-1$  and  from Theorem \ref{finalupperbounde2} (\ref{depthh2}), $r_{Q}(I)=\mathrm{g}_{s}(I)+1$. Further, again from  Theorem \ref{finalupperbounde2} (\ref{depthh3}),  $e_{i}(I)=\dbinom{\mathrm{g}_{s}(I)+1}{i}$ for $3 \leq i \leq d.$
\end{proof}

The following example illustrates the  bounds established on $e_{2}(I)$ in Theorem \ref{finalupperbounde2}. 

 \begin{example} \cite[Example 2.3]{sk} \label{optimalexample}
 Let $S=\mathbb Q[[X,Y,Z,W]]$ and $K = (X,Y)\cap(Z,W)= (XZ, XW, YZ, YW),$ where $X,Y,Z,W$ are variables. Put  $ R= S/K=\mathbb Q[[x,y,z,w]],$ where $x,y,z,w$ are  the images of  $X,Y,Z,W$ in $R.$ Then  $\dim R=2$  and $\depth R=1.$   Let $\m=(x,y,z,w)$ denote the maximal ideal of $R.$ Note that  $x-z$, $y-w$ is a system of parameters for $R$  and $Q=(x-z,y-w)$ is a minimal reduction of $\m.$  Then
 \begin{align*}
     \lambda \left( \frac{R}{\m^n }\right)
     &= \lambda \left( \frac{S} {(X,Y)^n + (Z,W)^n + (XZ, XW, YZ, YW)}\right)
     = 1 +   2(2 + \cdots + n)
     = 2 \binom{n+1}{2}-1 \text{ for } n\geq 1,\\ 
     \lambda \left( \frac{R}{Q^n }\right)
    & = \lambda \left( \frac{S} {(X^{n-i}Y^i +  (-1)^nZ^{n-i}W^i:i=0, \ldots,n) + (XZ, XW, YZ, YW)}\right)&\\
    & = 1 +  2(2 + \cdots + n) + (n+1)
     = 2 \binom{n+1}{2} +n  \text { for } n \geq 2. 
 \end{align*}
   Hence, $e_{0}(\m)=2$, $e_{1}(\m)=0$ and 
 $e_{2}(\m)=-1,$  $e_{0}(Q)=2,$  $e_{1}(Q)=-1$,  $\mathrm{g}_{s}(\m)=-1$  and $\mathrm{g}_{s}(Q)=0.$  From \cite[Corollary 2.9]{goto1980}, $\depth G(I)=1.$ Then
$$-1=e_{2}(\m) \geq  \mathrm{g}_{s}(\m)-\mathrm{g}_{s}(Q)+e_{1}(Q)\dbinom{e_{0}(\m)-e_{1}(I)}{2}=-1-\binom{3}{2}=-4.$$
 
 Further, Set $R'=R/(x-z)$, $\m'=\m R'$ and $Q'=QR'.$ Since $e_{1}(Q')=-\lambda(H^{0}_{\m'}(R'))=-1$ and
 $0<\lambda(U_{\m}(2;1))\leq \lambda(H^{0}_{\m'}(R')),$ thus, 
 $\lambda(U_{\m}(2;1))=1.$ Further,
\small{\begin{align}
    \label{u_m(2,2)} \nonumber
     U_{\m}(2;2) 
     &= \frac{ ((X-Z) + K: ((Y-W)  + K) )\cap ((X,Y,Z,W)^2 + (X-Z) + K)}{(X-Z) + K}\\
     &= \frac{  (X-Z, XZ, XW, YZ, YW): ((Y-W  , XZ, XW, YZ, YW) )\cap (X-Z, YW, Z^2, YZ, W^2,ZW,  Y^2)}{(X-Z,  XW, Z^2, YZ, YW)}.
\end{align}}

\normalsize
Now, $ (X-Z) + K: ((Y-W)  + K)=(X-Z, XZ, XW, YZ, YW): (Y-W  , XZ, XW, YZ, YW)=(X,Z, YW).$ Therefore, 
\footnotesize{\begin{align} \label{u_m(2,2)-2} \nonumber
    (X,Z, YW) \cap ((X-Z, YW, Z^2, ZW, YZ) +  (W^2,  Y^2))& =(X-Z, YW, Z^2, ZW, YZ) + (X,Z, YW) \cap (W^2,  Y^2)\\ \nonumber
  &  = (X-Z, YW, Z^2, ZW, YZ) + (ZW^2, YW^2, XW^2, Y^2W, Y^2Z, XY^2)\\ \nonumber
   & = (X-Z, YW, Z^2, ZW, YZ)\\
   &= (X-Z, YW, Z^2, XW, YZ). 
\end{align}}

\normalsize From (\ref{u_m(2,2)}) and (\ref{u_m(2,2)-2}), we get $\lambda(U_{\m}(2;2))=0.$
    Then $$-1=e_{2}(\m) \leq \binom{\mathrm{g}_s(\m) - \mathrm{g}_s(Q) - (e_{0}(\m)-e_{1}(Q))e_1(Q) + 1}{2}-\lambda(U_{\m}(2;1))=\binom{3}{2}-1=2.$$
  \end{example}
\normalsize

We now obtain an upper bound for $e_2(I)$ in the case $\dim R = 2$ and  $\depth R = 0$.

\begin{proposition} \label{lw} 
    Let $(R,\m)$ be a two-dimensional Noetherian local ring with $\depth R =0.$ Let $I$ be an $\m$-primary ideal with minimal reduction of $Q$.   Then
    \begin{equation*}
    \label{eqn:lw}
         e_{2}(I) 
\leq  \binom{\mathrm{g}_{s}(I)+(e_{0}(I)-e_{1}(Q))\left(1-e_{1}(Q)\right)}{2}
+       \lambda\left(H_{\m}^{0}(R)\right).
    \end{equation*}
\end{proposition}

\begin{proof}
Put  $\overline{R}=R/\left(H_{\m}^{0}(R)\right), \overline{I}=I\overline{R}$ and $\overline{Q}=Q\overline{R}$.   Then 
\begin{align} \label{ccc} \nonumber
    \mathrm{g}_{s}(\overline{I}) -  \mathrm{g}_{s}(\overline{Q})
 &= \lambda(\overline{R}/\overline{I})-e_{0}(\overline{I})+e_{1}(\overline{I}) \nonumber
- ( \lambda(\overline{R}/\overline{Q})-e_{0}(\overline{Q})+e_{1}(\overline{Q})) \\ \nonumber
&= \lambda\left(\frac{R}{I + H_{\m}^{0}(R)}\right)-e_{0}({I})+e_{1}({I})
   -\lambda\left(\frac{R}{Q + H_{\m}^{0}(R)}\right)+e_{0}({Q})-e_{1}({Q}) 
     \mbox{ (by \cite[Proposition 2.3]{rv})}\\ \nonumber
&\leq \lambda({R}/{I})-e_{0}({I})+e_{1}({I}) -1+e_{0}(Q)-e_{1}(Q)\\
&= \mathrm{g}_s(I) - 1+e_{0}({I})-e_{1}({Q}).
 \end{align}

Since $\overline{R}$ is a two-dimensional ring with $\depth \overline{R} >0.$ Therefore, from from \cite[Proposition 2.3]{rv}), we have 
\begin{align} \nonumber
e_{2}(I)&=e_{2}(\overline{I}))+\lambda\left(H_{\m}^{0}(R)\right) \\\nonumber
& \leq \binom{\mathrm{g}_{s}(\overline{I})-\mathrm{g}_{s}(\overline{Q})-(e_{0}(\overline{I})-e_{1}(\overline{Q}))e_{1}(\overline{Q})+1}{2}+\lambda\left(H_{\m}^{0}(R)\right) \text{ (by Proposition \ref{finalupperbounde2})}\\ \nonumber\label{ddd}
    &\leq \binom{\mathrm{g}_{s}(I)+(e_{0}(I)-e_{1}(Q))(1-e_{1}(Q))}{2}+\lambda\left(H_{\m}^{0}(R)\right) \text{ (by (\ref{ccc})) }. &&\qedhere
\end{align} 
\end{proof}

In the following proposition, we study the variation of $e_{2}(I)$, i.e., how $e_{2}(I)$ varies when $I$ is enlarged. Here, the reference point for this comparison is a parameter ideal $Q$ generated by a system of parameters of $R$, which is a minimal reduction of $I.$ More precisely, we bound the deviation $e_{2}(I)-e_2(Q)$ in terms of sectional genus of $I.$

\begin{proposition} \label{deviationbound} 
    Let $(R,\m)$ be a Buchsbaum local ring of dimension $d \geq 2$. Let $I$ be an $\m$-primary ideal with minimal reduction $Q$. Then
    \begin{eqnarray}   \label{eqn-mum}
    e_{2}(I)-e_{2}(Q) \leq \binom{\mathrm{g}_{s}(I)+(e_{0}(I)-e_{1}(Q))\left(1-e_{1}(Q)\right)}{2}.
    \end{eqnarray}
\end{proposition}
\begin{proof}
        We prove by induction on dimension $d$. 
        Let  $d=2$.  Since $Q$ is a parameter ideal,  by \cite[Corollary 4.2]{trung3}, we get $e_{2}(Q)=\lambda\left(H_{\m}^{0}(R)\right).$ Thus, the conclusion follows from Proposition \ref{lw}. 

        Assume $d \geq 3$ and that our assertion holds for $d-1.$ 
         Let $x\in Q$ be a superficial element of both $Q$ and $I.$ Set $R'=R/(x),$ $I'=IR'$ and $Q'=QR'.$  From Lemma \ref{lemmaa1}, we have $\mathrm{g}_{s}(I')=\mathrm{g}_{s}(I)$ and from Proposition \ref{hibertcoeffsup}, we have $e_{0}(I')=e_{0}(I)$ and $e_{1}(Q')=e_{1}(Q)$.
        Hence,
        \begin{align*}
            e_{2}(I)-e_{2}(Q)
            &=e_{2}(I')-e_{2}(Q') 
            \text{ (by Proposition \ref{hibertcoeffsup})}\\
            &\leq \binom{\mathrm{g}_{s}(I')+(e_{0}(I')-e_{1}(Q'))\left(1-e_{1}(Q')\right)}{2} \text{ (by induction hypothesis)}\\
            &=  \binom{\mathrm{g}_{s}(I)+(e_{0}(I)-e_{1}(Q))\left(1-e_{1}(Q)\right)}{2}.   \qedhere
        \end{align*}
\end{proof}

 We conclude this section with the following example  which provides an infinite class of examples where the bound in Proposition \ref{deviationbound} is sharp.
\begin{example} \cite[Example~3.6]{goto-yoshida}
\label{example depth 0} 
    Let $\kk$ any field, and let $d \geq 2$ and $r \geq 1$ be integers. Put $A = \kk[X_0, \ldots, X_d]$   and $K_r =  (X_0 (X_0, \ldots, X_d)^r)$. Then is Buchsbaum if and only if $r=1.$ Let 
     $R_{d,r}= (A/ K_r)_{\n (A/K_{r})}= \kk[x_0, \ldots, x_d]_{\m}$ where $\n=(X_0, \ldots, X_d)A$ and $\m =\n R_{d,r}.$
       Then we have the following:
    \begin{enumerate}
    \item 
    For all $r \geq 1$, $\depth R_{d,r} =0 = \depth G({\m})$. And  $R_{d,1}$  is locally Buchsbaum of dimension $d$ if and only if $r=1.$

        \item Let  $\q= (X_1, \ldots, X_d)$ and $Q = \q R_{d,r}$.    
        Since $K_r = \displaystyle  \sum_{i=0}^{r} (X_0^{r+1-i}) (X_1, \ldots, X_d)^i$,    
        for all $n \geq r+1$, we have     
 \begin{eqnarray} \label{Qreg12} \nonumber
     \n^n + K_r &=& (X_0, \ldots, X_d)^n + K_r\\ \nonumber
     &=& (X_{0})(X_1, \ldots, X_d)^{n-1}+(X_{1},\ldots,X_{d})^{n} +K_{r}\\ \nonumber
     &=& (X_1, \ldots, X_d)^r +K_{r}\\ 
&=&\q^n + K_r. 
 \end{eqnarray}
Therefore, for all $n \geq r+1$, we have the following short exact sequence 
 \begin{eqnarray} \label{seqabc}
 0 \longrightarrow
  \frac{(X_{1},\ldots,X_{d})^{n}+(X_{0})}{\n^{n}+K_{r}}
   \longrightarrow
      \frac{A}{\n^n + K_r}
      \longrightarrow
 \frac{A}{(X_1, \ldots, X_d)^n + (X_0)} \longrightarrow 0.
 \end{eqnarray}
Note that from equation (\ref{Qreg12}), we have 
$\dfrac{(X_{1},\ldots,X_{d})^{n}+(X_{0})}{\n^{n}+K_{r}}=\dfrac{(X_{1},\ldots,X_{d})^{n}+(X_{0})}{\q^{n}+K_{r}}$, thus, 
$$\frac{(X_{1},\ldots,X_{d})^{n}+(X_{0})}{\q^{n}+K_{r}} \cong \frac{(X_{0})}{(\q^{n}+K_{r})\cap(X_{0})}  = \frac{(X_{0})}{(X_{0})(X_{0},\ldots, X_{d})^{r}}  \cong \frac{A}{(X_{0},\ldots, X_{d})^{r}}.$$

Consider
 \begin{align*}
    \lambda \left( \frac{R_{d,r}}{\m^n}\right)
    &=  \lambda \left( \frac{A}{\n^n + K_r}\right) & \\
    &= \lambda \left(\frac{A}{(X_1, \ldots, X_d)^n + (X_0) } \right)
  +  \lambda \left( \frac{A}{(X_0, \ldots, X_d)^r} \right) & \text{ (from (\ref{seqabc}))}\\
  &= \dbinom{n+d-1}{d} + \dbinom{r+d}{d+1} & \text{(for sufficiently large n).}
 \end{align*}
 Hence 
\begin{eqnarray*}
e_i(\m) =e_i(Q)= \begin{cases}
   1 & ;i=0\\
   0 & ;i=1, \ldots, d-1\\
   (-1)^d\dbinom{r + d}{d+1} & ;i=d.  
\end{cases}
\end{eqnarray*}
\item 
If $d=2$, then $e_2(\m)=\dbinom{r+2}{3}>0$. Further, if $d>2$, then $e_2(\m)=0$ but $e_d(\m) \not =0$. This shows that  Corollary~\ref{lori} does not hold true without the assumption  that $\depth R \geq d-1$. 

\item Let $r \geq 1$. Then we have 
\begin{eqnarray*}
    H^0_{\m}(R_{d,r}) 
    \cong \frac{(X_0)}{X_0(X_0, X_1, \ldots, X_d)^r} 
    \cong  \frac{ A}{(X_0, X_1, \ldots, X_d)^r}. 
\end{eqnarray*}
Hence,  if $d=2$, then  $\lambda(H^0_{\m}(R_{d,r}) ) = e_2(\m),$ which shows that the bound of $e_2(\m)$ is sharp and  verifies  
Proposition~\ref{parameterbounddepthzero}.

\item Let $r=1,$ then $R_{d,r}$ is Buchsbaum. Note that $\mathrm{g}_s(\m) = 0$, thus, from (\ref{eqn-mum}) we have 
\begin{eqnarray*}
    0 = e_2(\m)-e_2(Q) \leq \binom{\mathrm{g}_s(\m) + (e_0(\m) - e_1(Q)) (1 -e_1 (Q))}{2}
    = \binom{0  + 1}{2} =0.
 \end{eqnarray*}
 This implies that the bound in Proposition \ref{deviationbound} is sharp. 
    \end{enumerate}
 \end{example}

\section{An upper bound fpr the Second Hilbert Coefficient in Buchsbaum rings} \label{section5}
In this section, we give another upper bound for $e_{2}(I)$ in a Buchsbaum local ring using the technique of $S_{2}$-fication.  Note that the upper bound given in Theorem \ref{finalupperbounde2}  depends on $\displaystyle \sum_{n=1}^{e_{0}(I)-e_{1}(Q)-1}
           n\lambda(U_{I}(d;n))$, which can be challenging to compute. In contrast, the following upper bound is easier to compute, as it relies solely on the sectional genus and Hilbert coefficients.

We recall the concept of $S_{2}$-fication of  generalized Cohen-Macaulay rings  (see  \cite[Section 4]{gghv}).  Let $R$ be a Noetherian ring with the total ring of fractions $\mathbb K$. The smallest finite extension $\phi: R\longrightarrow \mathtt S \subseteq \mathbb K$ satisfying Serre's condition $(S_{2})$ is called the $S_{2}$-fication of $R$. 
If  $(R,\m)$ be a Buchsbaum local ring with positive depth,   then $\mathtt S=\Hom_{R}(\m,\m)$ is the $S_{2}$-fication of $R$ ( see \cite[Theorem 4.2]{gghv}). The ring  $\mathtt S$ is semilocal   (\cite[Proposition 4.1(2)]{gghv}). Let $\{\M_{1},\ldots, \M_{l}\}$ be the maximal ideals of $\mathtt S,$ and let $\mathtt S_{1},\ldots,\mathtt S_{l}$ denote the corresponding localizations. Further, for each $1\leq i\leq l,$ let $f_{i}$ be the relative degree $[\mathtt S/\M_{i}:R/\m].$

In the following proposition, we  give another upper  bound for $e_{2}(I)$ under the hypothesis that $I=I\mathtt{S}.$
\begin{proposition} \label{dc}
    Let $(R, \m)$ be a two-dimensional Buchsbaum ring with  $\depth R >0$  and $\mathtt{S}$ be the $S_{2}$-fication of $R$. Let $I$ be an $\m$-primary ideal with minimal reduction $Q$ and assume that $I=I\mathtt{S}$. Then
\begin{equation*}
    e_{2}(I) \leq \binom{\mathrm{g}_{s}(I)-e_{1}(Q)+1}{2}+e_{1}(Q) .
    \end{equation*}
  \end{proposition}
\begin{proof}
        We first note that $\mathrm{g}_{s}(I\mathtt{S}_{i})=\lambda_{\mathtt{S}_{i}}(\mathtt{S}_{i}/I\mathtt{S}_{i})-e_{0}^{\mathtt{S}_{i}}(I\mathtt{S}_{i})+e_{1}^{\mathtt{S}_{i}}(I\mathtt{S}_{i})$ for all $i=1,\ldots,l.$ Now
    \begin{align} \nonumber
      \sum_{i=1}^{l}\mathrm{g}_{s}(I\mathtt{S}_{i})f_{i} 
 & =\sum_{i=1}^{l}[\lambda_{\mathtt{S}_{i}}(\mathtt{S}_{i}/I\mathtt{S}_{i})
                           -e_{0}^{\mathtt{S}_{i}}(I\mathtt{S}_{i})+e_{1}^{\mathtt{S}_{i}}(I\mathtt{S}_{i})]f_{i} & \\ \nonumber
 & = \sum_{i=1}^{l}\lambda_{\mathtt{S}_{i}}(\mathtt{S}_{i}/I\mathtt{S}_{i})f_{i}-\displaystyle \sum_{i=1}^{l}e_{0}^{\mathtt{S}_{i}}(I\mathtt{S}_{i})f_{i}
 +    \sum_{i=1}^{l}e_{1}^{\mathtt{S}_{i}}(I\mathtt{S}_{i})f_{i} &  \\ \nonumber
 & = \lambda_{R}(\mathtt{S}/I\mathtt{S})-e_{0}^{R}(I\mathtt{S})+e_{1}^{R}(I\mathtt{S}) & \text{ (from the proof of \cite[Theorem 4.3]{gghv})}\\ \nonumber
& =\lambda_{R}(R/I\mathtt{S})+\lambda_{R}(\mathtt{S}/R)-e_{0}^{R}(I\mathtt{S})+e_{1}^{R}(I\mathtt{S}) &\\ \nonumber
&=\lambda(R/I)-e_{1}(Q)-e_{0}(I)+e_{1}(I) &\text{(since $I=I\mathtt{S}$ and by  \cite[Lemma 4.4]{shruti})}\\ \label{dcx}
 &=\mathrm{g}_{s}(I)-e_{1}(Q). 
    \end{align}
From Corollary \ref{cml}, we have $e_{2}(I\mathtt{S}_{i}) \leq \dbinom{\mathrm{g}_{s}(I\mathtt{S}_{i})+1}{2}.$ Now from \cite[Lemma 4.2 (iii)]{shruti} and \cite[Lemma 4.4]{shruti}, we have
    \begin{align*} 
      e_{2}(I)
 & = e_{2}^{R}(I\mathtt{S})+e_{1}(Q) & \\\nonumber
 &=\displaystyle \sum_{i=1}^{l}e_{2}^{\mathtt{S}_{i}}(I\mathtt{S}_{i})f_{i}+e_{1}(Q) & \\ \nonumber
& \leq \sum_{i=1}^{l}\left(\frac{(\mathrm{g}_{s}(I\mathtt{S}_{i}))^2+\mathrm{g}_{s}(I\mathtt{S}_{i})}{2}\right)f_{i}+e_{1}(Q) & \\ \nonumber
 & \leq\frac{1}{2}\left[ \sum_{i=1}^{l} \left(\mathrm{g}_{s}(I\mathtt{S}_{i})f_{i}\right)^{2}
 +\sum_{i=1}^{l}\left( \mathrm{g}_{s}(I\mathtt{S}_{i})f_{i}\right)\right]+e_{1}(Q) & \\ \nonumber
& = \frac{(\mathrm{g}_{s}(I)-e_{1}(Q))^{2}+(\mathrm{g}_{s}(I)-e_{1}(Q))}{2}+e_{1}(Q)  \text{ (by Equation (\ref{dcx}))}.
    \end{align*}
    Therefore, $e_{2}(I) \leq \dbinom{\mathrm{g}_{s}(I)-e_{1}(Q)+1}{2}+e_{1}(Q)$.
\end{proof}

As an immediate consequence, we obtain the main result of this section, which provides an upper bound for the second Hilbert coefficient of the maximal ideal of a Buchsbaum local ring with depth at least $d-1$.
\begin{theorem} \label{s2ficationmaximal}
  Let $(R,\m)$ be a Buchsbaum local ring of dimension $d \geq 2$ with $\depth R\geq d-1$. Let $Q$ be a minimal reduction of $\m$, then   $e_{2}(\m) \leq \dbinom{\mathrm{g}_{s}(\m)-e_{1}(Q)+1}{2}+e_{1}(Q).$
\end{theorem}
\begin{proof}
 We prove by induction on the dimension $d.$ For $d=2,$ let $\mathtt{S}$ be the $S_{2}$-fication of $R.$
Using \cite[Proposition 4.1]{gghv}  and  from \cite[Chapter 2, Proposition 2.1 (iii)]{sv}, $\mathfrak{m} H
    ^{1}_{\mathfrak{m}}(R) = 0$, we have that  $\m =\m\mathtt{S}.$ Therefore, the conclusion follows directly from Proposition \ref{dc}.

 Assume $d \geq 3$ and that our assertion holds for $d-1.$  Let $x$ be a superficial element for $\m.$ Set $R'=R/(x)$ and $\m'=\m R'.$ Then $R'$ is a Buchsbaum ring of dimension $d-1$ with $\depth R' \geq d-2.$ Hence,
 \begin{align*}
     e_{2}(\m)
     &=e_{2}(\m')  \text{ (by Proposition \ref{hibertcoeffsup})}\\
     & \leq \dbinom{\mathrm{g}_{s}(\m')-e_{1}(Q')+1}{2}+e_{1}(Q')  \text{ (from induction hypothesis)}\\
     &=\dbinom{\mathrm{g}_{s}(\m)-e_{1}(Q)+1}{2}+e_{1}(Q) \text{ (by Proposition \ref{hibertcoeffsup} and Lemma \ref{lemmaa1}).} &\qedhere
 \end{align*}
\end{proof}

The following example illustrates Proposition \ref{dc}.

\begin{example} 
   Let $S=\mathbb K[[X,Y,Z,W]],$ where $X,Y,Z,W$ are variables and let $\mathfrak a =(X,Y)$ and $\mathfrak b= (Z,W)$ be two ideals in $S.$     Let $R=S/\mathfrak a \cap \mathfrak b,$ then $R$ is a two-dimensional Buchsbaum ring with $\depth R=1.$   Further, let $x,y,z,w$ denote the images of $X,Y,Z,W$ in $R$ and $\m=(x,y,z,w)$ be the maximal ideal. We have the following short exact sequence:
   \begin{equation} \label{seqs}
       0\longrightarrow R \longrightarrow \dfrac{S}{\mathfrak a}\bigoplus \dfrac{S}{\mathfrak b} \longrightarrow  \dfrac{S}{\mathfrak a +\mathfrak b}   \longrightarrow 0.
   \end{equation}
   Let $I=I_{1}+I_{2}$ be an $\m$-primary ideal in $R$ where $I_{1} \subseteq (x,y)$ and $I_{2} \subseteq (z,w).$ Note that the $S_{2}$-fication $\mathtt{S}$ of $R$ is $\dfrac{S}{\mathfrak a}\bigoplus \dfrac{S}{\mathfrak b},$ so that $I=I\mathtt{S}.$  Tensoring (\ref{seqs}) with $R/I^{n},$ we get:
    \begin{equation*}
         0 \longrightarrow \frac{I^{n}\mathtt{S} \cap R}{I^{n}} \longrightarrow \dfrac{R}{I^{n}}\longrightarrow \frac{S}{\mathfrak a + I^{n}}\bigoplus \frac{S}{\mathfrak b+ I^{n}}\longrightarrow \frac{S}{\mathfrak a+ \mathfrak b+ I^{n}}\longrightarrow 0.
     \end{equation*}
     Since $I=I \mathtt{S}$, we have $\dfrac{I^{n}\mathtt{S} \cap R}{I^{n}}=0$ for all $n.$ Thus, we get:
     \begin{equation} \label{length}
         0 \longrightarrow R/I^{n}\longrightarrow \frac{S}{\mathfrak a + I^{n}}\bigoplus \frac{S}{\mathfrak b+ I^{n}}\longrightarrow \frac{S}{\mathfrak a+ \mathfrak b+ I^{n}}\longrightarrow 0.
     \end{equation}
     Thus, from the exact sequence (\ref{length}), we get $\lambda(R/I^{n})+\lambda_{R}(S/(\mathfrak a+ \mathfrak b+ I^{n}))=\lambda_{R}(S/\mathfrak a + I^{n}))+\lambda_{R}(S/( \mathfrak b+ I^{n})).$
     Note that $I^{n} \subseteq \mathfrak a+\mathfrak b$ for sufficiently large $n$.  Thus, for sufficiently large $n,$ we have the following equation:
    \begin{equation} \label{november}
      \displaystyle \sum_{i=0}^{2}\left[(-1)^{i}e_{i}(I)\binom{n+d-1-i}{d-i}\right]+1=\displaystyle \sum_{i=0}^{2}(-1)^{i}\left[e_{i}(I,S/\mathfrak a) +e_{i}(I,S/\mathfrak b)\right]\binom{n+d-1-i}{d-i}.
     \end{equation}
     On comparing the coefficients in (\ref{november}), we get 
     $e_{2}(I)+1=e_{2}(I,S/\mathfrak a)+e_{2}(I,S/\mathfrak b).$ Furthermore, from \cite[Example 4.7]{gghv} for any parameter ideal $Q$ in $R$,   $e_{1}(Q)=-1$, and
     $\lambda_{R}(S/(\mathfrak a+\mathfrak b))=\lambda\left(H^{1}_{\m}(R)\right)=-e_{1}(Q)$ because $R$ is Buchsbaum \cite[Chapter 1, Propositions 2.6 and 2.7]{sv}. 
         Thus,  $e_{2}(I)=e_{2}(I,S/\mathfrak a)+e_{2}(I,S/\mathfrak b)+e_{1}(Q).$  Consider, 
         \begin{align} \nonumber
           e_{2}(I,S/\mathfrak a) +e_{2}(I,S/\mathfrak b) &\leq   \binom{\mathrm{g}_{s}(I,S/\mathfrak a)+1}{2}+ \binom{\mathrm{g}_{s}(I,S/\mathfrak b)+1}{2} \text{ (from Corollary \ref{cml})}\\ \nonumber
             &= \frac{\mathrm{g}_{s}(I,S/\mathfrak a)(\mathrm{g}_{s}(I,S/\mathfrak a)+1)+\mathrm{g}_{s}(I,S/\mathfrak b)(\mathrm{g}_{s}(I,S/\mathfrak b)+1)}{2}\\ \nonumber
                       & \leq \frac{(\mathrm{g}_{s}(I,S/\mathfrak a)+\mathrm{g}_{s}(I,S/\mathfrak b))^{2}+\mathrm{g}_{s}(I,S/\mathfrak a)+\mathrm{g}_{s}(I,S/\mathfrak b)}{2} \\ \label{ddu}
            & = \binom{\mathrm{g}_{s}(I,S/\mathfrak a)+\mathrm{g}_{s}(I,S/\mathfrak b)+1}{2}.
         \end{align}
        Note that \begin{align} \nonumber
            \mathrm{g}_{s}(I,S/\mathfrak a)+\mathrm{g}_{s}(I,S/\mathfrak b)&= \lambda_{R}\left(\frac{S}{\mathfrak a+I}\right)-e_{0}(I,S/\mathfrak a)+e_{1}(I,S/\mathfrak a)+\lambda_{R}\left(\frac{S}{\mathfrak b+I}\right)-e_{0}(I,S/\mathfrak b)+e_{1}(I,S/\mathfrak b)\\ \nonumber
           &=\lambda(R/I)+\lambda_{R}\left(\frac{S}{\mathfrak a+ \mathfrak b+I}\right)-e_{0}(I)+e_{1}(I)\\
           & \leq \lambda(R/I)+\lambda_{R}\left(\frac{S}{\mathfrak a+ \mathfrak b}\right)-e_{0}(I)+e_{1}(I)\\ \nonumber
             & = \lambda(R/I)-e_{1}(Q)-e_{0}(I)+e_{1}(I)\\ \label{indra}
             & =\mathrm{g}_{s}(I)-e_{1}(Q).
        \end{align}
        From (\ref{ddu}) and (\ref{indra}), we get 
        $e_{2}(I) \leq \dbinom{\mathrm{g}_{s}(I)-e_{1}(Q)+1}{2}+e_{1}(Q).$
\end{example}
We conclude this section by recalling Example \ref{optimalexample} which  illustrates that the upper bound for  $e_{2}(\m)$ obtained in Theorem \ref{s2ficationmaximal} is sharp.
\begin{example}
    From Example \ref{optimalexample}, we have $\mathrm{g}_{s}(\m)=-1$, $e_{2}(\m)=-1$ and $e_{1}(Q)=-1.$ Therefore, from Theorem \ref{s2ficationmaximal}, we get $$-1=e_{2}(\m)\leq\binom{\mathrm{g}_{s}(\m)-e_{1}(Q)+1}{2}+e_{1}(Q)=-1.$$
\end{example}

\section{Hilbert Coefficients and the Depth of the Associated Graded Ring} \label{section6}
Throughout this section, let $(R, \m)$  be a Cohen-Macaulay local ring,   with an  infinite residue field and $I$   be an $\m$-primary ideal of $R$. 
In this section, we obtain interesting results on the interplay of the vanishing of certain Hilbert coefficients and the depth of the associated graded ring. 
We begin with  the vanishing of  $e_{2}(\mathcal{I})$ in  Cohen-Macaulay local rings of dimension $d \geq 2$.  

\begin{proposition} \label{e2equalszeroCM}
    Let $(R, \m)$ be a Cohen-Macaulay local ring of dimension $d \geq 2,$ $I$ an $\m$-primary ideal and $\mathcal{I}=\{I_{n}\}$ be an $I$-admissible filtration. Suppose $e_{2}(\mathcal{I})=0,$ then $\depth G(\mathcal{I}) \neq d-1.$
\end{proposition}

\begin{proof}
     Let us assume the contrary. 
    Suppose  $\depth G(\mathcal{I}) = d-1,$ then  from  \cite[Theorem 2.5]{rv}, we have 
    \begin{equation} \label{expressionfore2}
0 =  e_{2}(\mathcal{I})
 =\displaystyle \sum_{n \geq 1} 
        \binom{n}{1}\lambda\left(\frac{I_{n+1}}{JI_{n}}\right).
    \end{equation}
Hence   $I_{n+1}=JI_{n}$ for all $n \geq 1,$ which implies that  $r_{J}(\mathcal{I}) \leq 1$.  Thus from   \cite[Theorem 2.5]{rv}, $e_{0}(\mathcal{I})-e_{1}(\mathcal{I})=\lambda(R/J)-\lambda(I_{1}/J)=\lambda(R/I_{1})$. From \cite[Theorem 2.9(3)]{rv}, $G(\mathcal{I})$ is Cohen-Macaulay, which leads to a contradiction.  Hence $\depth G(\mathcal{I}) \neq d-1.$
\end{proof}

As a consequence, we have the following corollary for the $I$-adic case. We also recover a version of  Narita's result \cite[Theorem 1]{narita} on the $\depth G(I^{n})$ for sufficiently large $n$.

\begin{corollary} \label{core2equalszeroCM}
     Let $(R, \m)$ be a Cohen-Macaulay local ring of dimension $d \geq 2,$ and $I$ an $\m$-primary ideal. Suppose $e_{2}(I)=0.$ Then 
     \begin{enumerate}[\normalfont(i)]
        \item \label{papa-1}
        $\depth G(I) \neq d-1,$ 

        \item \label{papa-3}\cite[Theorem 1]{narita}
        If $d=2$, then  $G(I^{n})$ is Cohen-Macaulay for sufficiently large $n.$ 
        
\item \label{papa-2}
Let $d \geq 3.$ Suppose  $e_3(I) \not =0$. Then 
\begin{enumerate}[\normalfont(a)]
\item 
\label{papa-2-1}
  $\depth G(I) \leq  d-2$,

\item \label{papa-2-2}
$1 \leq \depth \widetilde{G}(I)\leq d-2.$
\end{enumerate}
\end{enumerate}
\end{corollary}
\begin{proof}
 (\ref{papa-1}) Follows directly from Proposition \ref{e2equalszeroCM}. 
    
 (\ref{papa-3}) From \cite[Lemma 2.12]{tjm}, we have $e_{2}(I)=e_{2}(I^{n})=0$ for all $n \geq 1.$ Since  $\depth G(I^{n}) >0$ for all sufficiently large $n$,  thus, the conclusion follows from (i).

 (\ref{papa-2})(a) By  (\ref{papa-1}),  we have
 $\depth G(I) \neq d-1.$  If   $\depth G(I) = d$, then since $e_2(I)=0$,  by \cite[Corollary~2 (2)]{marley},   we get  $e_3(I)=0$, which leads to a contradiction. Therefore, $\depth G(I) \leq d-2.$
 
  (\ref{papa-2})(b) Suppose  $\depth \widetilde{G}(I)\geq d-1.$ From \cite[Theorem 2.5]{rv}, 
    we have 
    $\widetilde{e}_{2}(I)=\displaystyle \sum_{n\geq1}\binom{n}{1}\lambda\left(\frac{\widetilde{I^{n+1}}}{J\widetilde{I^{n}}}\right).$ Since $\widetilde{e}_{2}(I) = e_2(I) =0$, we get   that 
    $\widetilde{I^{n+1}}= J\widetilde{I^{n}}$ for all $n \geq 1.$ Again from \cite[Theorem 2.5]{rv},
    we have $\widetilde{e}_{3}(I)=0$ which   implies that   $e_{3}(I)=\widetilde{e}_{3}(I)=0$ and this leads to a  contradiction. Therefore,  $\depth \widetilde{G}(I)\leq d-2.$ \qedhere
\end{proof}

As an immediate corollary, we recover Itoh's result \cite[Corollary 5]{itoh}

\begin{corollary}\cite[Corollary 5]{itoh} Let $(R, \m)$ be a two dimensional Cohen-Macaulay local ring and $I$ $\m$-primary ideal.
Then the following assertions are equivalent.
\begin{enumerate} [\normalfont(i)]
    \item $e_{1}(I)-e_{0}(I)+\lambda(R/\widetilde{I})=0.$
    \item $\widetilde{I}^{2}=J\widetilde{I}.$
    \item $\widetilde{I^{2}}=J\widetilde{I}.$
    \item $\widetilde{I^{n+1}}=J^{n}\widetilde{I}$ for every $n \geq 1.$
    \item $e_{2}(I)=0.$
\end{enumerate}
\end{corollary}
\begin{proof}
Note  that (iv)$\iff$(v) follows directly from equation (\ref{expressionfore2}).  We now show (v)$\iff$(i). Suppose $e_{2}(I)=0.$ Note that  $e_{i}(I)=e_{i}(\mathcal{I})$ for $0\leq i\leq 2$, where $\mathcal{I}=\{\widetilde{I^{n}}\}_{n\geq1}$. From Proposition \ref{e2equalszeroCM}, we have that $\depth G(\mathcal{I}) \neq 1$ as  $0=e_{2}(I)=e_{2}(\mathcal{I}).$ Since $\depth G(\mathcal{I}) \geq 1,$ thus, $G(\mathcal{I})$ is Cohen-Macaulay, this implies $e_{1}(I)-e_{0}(I)+\lambda(R/\widetilde{I})=0.$ Conversely, if $e_{1}(I)-e_{0}(I)+\lambda(R/\widetilde{I})=0,$ then $\widetilde{I^{n+1}}=J\widetilde{I^n}$ for all $n \geq 1.$ Thus, $e_{2}(I)=0.$
From \cite{huneke}, we have (i)$\iff$(ii).  The implication (iv)$\implies$(iii) follows trivially. For the implication (iii)$\implies$(ii), note that $J\widetilde{I} \subseteq \widetilde{I}^{2}\subseteq \widetilde{I^{2}}=J\widetilde{I}.$ Therefore, $\widetilde{I^{2}} =J\widetilde{I}.$
\end{proof}

\allowdisplaybreaks
The following example gives us an infinite class of ideals $I$ satisfying    $e_{2}(I)=0$ and $G(I^{n})$ is Cohen-Macaulay for sufficiently large $n.$

\begin{example} 
\label{example e2=0}
Let  $R =\kk [[x,y]]$, where $x,y$ are variables and $\mathbb K$ is any field.  Let   $r \geq 4$ be any integer,  and let $I_{r} = (x^r, x^{r-1}y, xy^{r-1}, y^r)$. 
Then the following are true:
\begin{enumerate} [(i)]
\item \label{example e2=0-1}$e_{2}(I_{r})=0,$
\item \label{example e2=0-2} $\depth G(I_{r})=0,$
\item \label{example e2=0-3}$G(I_{r}^{n})$ is Cohen-Macaulay for sufficiently large $n.$
  \end{enumerate}
  \end{example}
  \begin{proof}
 (\ref{example e2=0-1}) Note that for all $n \geq r-1, $ we have
    \begin{align*}
        I^{n}_{r}&=(x,y)^{n}(x^{r-1}+y^{r-1})^{n}
        = (x,y)^{n}\left(\displaystyle \sum_{i=0}^{n}(x^{r-1})^{(n-i)}(y^{r-1})^{i}\right)
         =(x,y)^{n}(x,y)^{(r-1)n} = (x,y)^{nr}.
    \end{align*}
    Therefore, for all $n \geq r-1$,
    \begin{eqnarray*}
    \lambda \left( \frac{R}{I^{n}_{r}}\right)
    = \lambda \left( \frac{R}{(x,y)^{nr}}\right)
     = \binom{nr + 1 }{2}
     = r^2 \binom{n+1}{2} - \binom{r}{2}n
  \end{eqnarray*}
  Hence $e_2(I_r)=0$. 
  
(\ref{example e2=0-2}) We claim that  $x^{r-2}y^{r-2} \in \widetilde{I_{r}} \setminus I_{r}$. One can verify that 
 \begin{eqnarray*}
      x^r ( x^{r-2}y^{r-2}) = (x^{r-1}y)^2 y^{r-4} \in I_{r}^2\\
      x^{r-1}y ( x^{r-2}y^{r-2}) = x^{r} (x y^{r-1}) x^{r-4} \in I_{r}^2.
       \end{eqnarray*}
       By symmetry, $y^r ( x^{r-2}y^{r-2}) \in I_{r}^2$ and $xy^{r-1} ( x^{r-2}y^{r-2})  \in I_{r}^2.$ Therefore, 
       $x^{r-2} y^{r-2}I_{r} \subseteq I_{r}^2$. 
Hence $x^{r-2}y^{r-2} \in  (I_{r}^2 : I_{r}) \subseteq \widetilde{I_{r}},$ but $x^{r-2}y^{r-2} \notin I_{r},$ which implies that   $\widetilde{I_{r}} \not = I_{r}$. Thus, $\depth G(I_{r}) = 0$. 

 (\ref{example e2=0-3}) Note that for all $k \geq r-1,$ we have $(I_r^{k+1}: I^k) = (x,y)^{r},$ which gives $\widetilde{I_r} = (x,y)^r,$ this implies $\widetilde{I_{r}}$ is a complete ideal. Therefore, $\widetilde{I_{r}^n}=\left(\widetilde{I_{r}}\right)^{n}=(x,y)^{nr}$ for all $n \geq 1.$  Hence from   \cite[Corollary~5.4]{lipman-teissier} and  \cite[Theorem~4.2]{john-verma}, $G\left(\left(\widetilde{I_{r}}\right)^{n}\right)$ is Cohen-Macaulay for all $n \geq 1$.  Furthermore, $\left(\widetilde{I_{r}}\right)^{n}=\widetilde{I_{r}^{n}}=I_{r}^{n}$ for $n$ large. Thus, $G(I_{r}^{n})$ is Cohen-Macaulay for sufficiently large $n.$ \qedhere
\end{proof}

By  examples we show that  for $d=3$, if $e_2(I)=0$ and $e_3(I) \neq 0$ for any $\mathfrak m$-primary ideal $I,$ then  $\depth G(I)$ can be either $0$ or $1$.  
\begin{example}
\label{e2=0, e3 not 0}
Let  $R =\kk [[x,y,z]],$ where $x,y,z$ are variables and $\mathbb K$ is any field.
   Let $I =  (x^3-y^3,  y^{3}-z^{3}, yz, xz, xy^{2})$. Then the following are true:
   \begin{enumerate} [\normalfont (i)]
     \item
\label{e2=0, e3 not 0-1}        $\depth G(I) = 0$. 

       \item
         \label{e2=0, e3 not 0-2}
       Let $K = I + (x^3) $ and $L = K +  (x^2y) = (x,y)^3 + z(x,y,z^2)$. Then 
   \begin{enumerate}
      \item 
   \label{e2=0, e3 not 0-2-a}
       $L = \overline{I} = \overline{K}$. 
        \item
           \label{e2=0, e3 not 0-2-b}
        $e_0(L) =15 = e_0(I) = e_0(K)$,  $e_1(L)= e_1(I) = e_1(K) = 7 $,  $e_2(L) = e_2(I) = e_2(K)= 0$,  $e_3(L) =0$ and $e_3(I) = e_3(K) =-1$.

        \item
         \label{e2=0, e3 not 0-2-c}
       $K = \widetilde{I}$,   $\depth~G(K)=1$ and  $G(L)$ is Cohen-Macaulay. 
     \end{enumerate}
   \end{enumerate}
  \end{example}
  \begin{proof}
 (\ref{e2=0, e3 not 0-1}) 
 Since $x^3 \in (I^3: I^2)\backslash I$, therefore,  $I$ is a not Ratliff-Rush closed. Hence $\depth G(I)=0$.  
  
 (\ref{e2=0, e3 not 0-2})(a) Let $S_1 = k[x, y_1, z_1]_{(x,y_1, z_1)}$ where $y_1 = y/z$ and $z_1 = z/x$. Then $I S_1  \cap R = (x^3, xz)S_1 \cap R = \m^3 + z \m$. Let  $S_2 = k[x_1, y, z_1]_{(x_1, y, z_1)}$. Then $IS_2 \cap R = (y^3 , yz)S_2 \cap R = \m^3 + z\m$.  Let  $S_3 = k[x_1, y_1, z]_{(x_1, y_1, z)}$. 
 Then $I S_3 \cap R = (xz, yz, z^3) \cap R = \m^3 + (x,y) \m$.  
Hence  $L =  \cap_{i=1}^{3}(IS_i \cap R = ((x,y,z)^3+ z (x,y,z)) \cap ((x,y,z)^3 + (x,y)(x,y,z))$.  Since $IS_i \cap R$ are integrally closed  ideals for all $i=1,2,3$ we get that $L$   is integrally closed. One can verity that $I \subseteq L \subseteq \overline {K} = \overline{I}$ which implies that  $L = \overline{I}$.

(\ref{e2=0, e3 not 0-2})(b)  By induction on $n$, we can show the following:  
 \begin{eqnarray} 
 \label{powers of K}
     K^n 
  &= &\sum_{i=0}^{n-1} z^{3n-2i} (x,y)^i + z^n    (x,y)^n
     + \sum_{i=1}^{n-1} z^{n-i} (x,y)^{n+2i}
     + (x^3, xy^2, y^3)^n\\ \label{powers of L}
          L^n 
& =&\sum_{i=0}^{n-1} z^{3n-2i} (x,y)^i   + z^n (x,y)^n
     + \sum_{i=1}^{n-1} z^{n-i} (x,y)^{n+2i} + (x,y)^{3n}\\ \label{powers of ln plus z}
     L^n + (z)
&= &(x,y)^{3n} + (z). 
\end{eqnarray}

Put $ L_{1,n} 
:= \left(  (z^n)
     + \displaystyle\sum_{i=1}^{n} (z^{n-i} x^{n+2i}) \right)$ and $ L_{2,n}
:=  \left(  (x^n) + \displaystyle \sum_{i=0}^{n-1} (z^{3n-2i} x^i)
    \right).$

    Then  we have
     \begin{eqnarray*}
 \left(  L_{1,n} + (y) \right) + 
     \left(  L_{2,n}    + (y) \right)
     =(z^n, x^n,y)
   \end{eqnarray*} 
   
    Consider  $L_{1,n}$ and $L_{2,n}$ ideals in $k[x,z]_{(x,z)}$. Then  both $L_{1,n}$ and $L_{2,n}$ are  contracted ideals of $\m$-adic order $n$,
    \cite[Proposition~2.3]{huneke-complete}. Put $z_1 = z/x$ and $T_1= k[x,z_1]_{(x,z_1)}$.
   Then   the strict transform of $L_{1,n}$  in $T_1$ is $L_{1,n}^{T_1} = (z_1 ^n) + \sum_{i=1}^{n} (z_1^{n-i} x^{n+i})$.  Put  $z_2 = z_1/x$ and $T_2= k[x,z_2]_{(x,z_2)}$.
   Then   the strict transform of $L_{1,n}^{T_2}$  in $T_2$ is $L_{1,n}^{T_2} = (z_2^n, x^n)$. 
Hence  on can apply the proof of  Lemma~3.9 in \cite{john-verma} for  the first quadratic transform of a contracted ideal to get
\begin{eqnarray*}
    \lambda \left( \frac{k[x,z]_{(x,z)}}{L_{1,n}}\right) 
=  \binom{n+1}{2} +  \lambda \left( \frac{T_1}{L_{1,n}^{T_1}}\right) 
= 2 \binom{n+1}{2} + \lambda \left( \frac{T_2}{L_{1,n}^{T_2}}\right) 
= 2 \binom{n+1}{2} + n^2.
 \end{eqnarray*}
Similarly one can show that 
 \begin{eqnarray*}
    \lambda \left( \frac{k[x,z]_{(x,z)}}{L_{2,n}}\right) 
     = 2 \binom{n+1}{2} + n^2.
    \end{eqnarray*}
    Since
 \begin{eqnarray}
    \label{formula L^n} \nonumber
  L^n + (y) 
 &=&\left(  \sum_{i=0}^{n-1} (z^{3n-2i} x^i)  + (z^n x^n) 
+     \sum_{i=1}^{n} (z^{n-i} x^{n+2i}) + (y) \right)
     \\ \nonumber
&=& \left(  (z^n)
 +    \sum_{i=1}^{n} (z^{n-i} x^{n+2i}) + (y) \right)
     \cap \left(  (x^n) +  \sum_{i=0}^{n-1} (z^{3n-2i} x^i) + (y) \right)\\ 
  &=&   (L_{1,n} +  (y) ) \cap (L_{2,n} + (y)),
      \end{eqnarray}
we get 
\begin{align}
\label{length ln+y}  \nonumber
    \lambda \left( \frac{R}{L^{n} + (y)} \right) 
    &= \lambda \left( \frac{R}{L_{1,n} + (y)} \right)  
    + \lambda \left( \frac{R}{L_{2,n} + (y)} \right) 
    - \lambda \left( \frac{R}{ L_{1,n} + L_{2,n} + (y)} \right)\\  \nonumber
    &= 2 n^2 + 4\binom{n+1}{2} - n^2\\
    &= n^2 +  4\binom{n+1}{2}.
\end{align}
     From  (\ref{powers of K}) and (\ref{powers of L}) we get $L^n = K^n + (x^{3n-1}y)$. Hence  it is enough to compute $\lambda(R/L^n)$. 
 For all $n \geq 2$,  \normalsize
\begin{eqnarray} \nonumber
     L^n : (yz) 
     &=& (z^{3n} :(yz)) + \sum_{i=1}^{n-1} (z^{3n-2i} (x,y)^{i} : (yz) 
     +  (z^n (x,y)^n : (yz))\\ \nonumber
     & &\quad
     + \sum_{i=1}^{n-1} (z^{n-i} (x,y)^{n+2i}  :(yz))
     + ((x,y)^{3n} :(yz)) \\ \nonumber
     &=& (z^{3n-1}) +  \sum_{i=1}^{n-1} z^{3n-2i-1} (x,y)^{i-1}
     + z^{n-1} (x,y)^{n-1}
     + \sum_{i=1}^{n-1} (z^{n-i-1} (x,y)^{n+2i-1} 
     + (x,y)^{3n-1}\\ \nonumber
     &=&  \sum_{i=0}^{n-2} z^{3n-2i-3} (x,y)^{i}
    + z^{n-1} (x,y)^{n-1}
    + \sum_{i=1}^{n-2} (z^{n-i-1} (x,y)^{n+2i-1}  + (x,y)^{3n-3}   + (x,y)^{3n-1}\\ \nonumber
    &=& \sum_{i=0}^{n-2} z^{3(n-1)-2i} (x,y)^{i}
    + z^{n-1} (x,y)^{n-1}
    + \sum_{i=1}^{n-2} (z^{(n-1)-i} (x,y)^{(n-1)+2i}  
    +  (x,y)^{3(n-1)}. \\ \nonumber
     &=& L^{n-1}.
 \end{eqnarray}
 Hence for all $n \geq 2$ we have the exact sequence,
\begin{eqnarray}
\label{ses L^n}
    0 \longrightarrow \frac{R}{L^{n-1}}
    \longrightarrow \frac{R}{L^{n}}
    \longrightarrow \frac{R}{L^{n-1} + (yz)}
    \longrightarrow 0.
\end{eqnarray} 
Hence for all $n \geq 1$,  
\begin{align}
\label{length of Ln} \nonumber
\lambda \left( \frac{R}{L^{n}} \right) 
&= \sum_{i=1}^n  \lambda \left( \frac{R}{L^{i} + (yz)} \right) & \mbox{(by (\ref{ses L^n}))}\\ \nonumber
&=\sum_{i=1}^n  \left[ \lambda \left( \frac{R}{L^{i} + (y)} \right) 
+   \lambda \left( \frac{R}{L^{i} + (z)} \right)
-  \lambda \left( \frac{R}{L^{i} + (y,z)} \right)\right]&\\ \nonumber
&= \sum_{i=1}^n \left[ i^2 + 4 \binom{i+1}{2} 
+ \binom {3i+1}{ 2 } -3i \right]  & \mbox{(by (\ref{length ln+y}), (\ref{powers of ln plus z}))}  \\ \nonumber
&= \sum_{i=1}^n \frac{15 i^2 + i}{2}& \\ \nonumber
&= \frac{n (n+1) (5n+3)}{2} & \\
& =  15 \binom{n+2}{3} - 7 \binom{n+1}{2}.  
   \end{align}
 From (\ref{powers of K}) and (\ref{powers of L}) we get $L^n= K^n + (x^{3n-1}y)$ for all $n \geq 1$. 
Moreover,  $I^2 = K^2 + (yz^4, z^6)$ and $I^n = K^n$ for all $n\geq 3$. Applying to (\ref{length of Ln}) we get that for all $n \geq 3$, 
\begin{align*}
 \lambda \left(  \frac {R}{I^n}\right)
= \lambda \left(  \frac {R}{K^n}\right)
= \lambda \left(  \frac {R}{L^n}\right)+1
 =  15 \binom{n+2}{3} - 7 \binom{n+1}{2} +1.   
  \end{align*}
  Since $L = \overline{K} = \overline{L}$, we get $e_0(L)= e_0(K) = e_0(I) =15$. Moreover, 
   $e_1(I) =e_1(K) = e_1(L) = 7$,  $e_2(I)=e_2(K)=0$ and $e_3(I) = e_3(K)=-1$. 

 (\ref{e2=0, e3 not 0-2})(c)
 If $\depth G(K)=2$, then  $e_2(I)=0$ implies  $e_3=0$   by \cite[Corollay~2(2)]{marley}, which leads to a contradiction as $e_3(I) =-1$. Hence $\depth G(K)=1$. 

As $e_2(L)=e_3(L)=0$,   by \cite[Theorem~2.1]{huneke}, the reduction number of $L$ is $1$.  Hence $\depth G(L)=3$.
  \end{proof}

\begin{example}  \label{example4.6}
Let  $R =\kk [[x,y,z]]$, where $x,y,z$ are variables and $\mathbb K$ is any field.
  Let $I =  (x^4-y^4,  y^{5}-z^{5}, yz, xz^2, xy^2)$. Then the following are true:
   \begin{enumerate} [\normalfont(i)]
    
     \item  \label{example4.6-1}
      $e_2(I) = 0$ and $e_3(I) =-1$. 
       \item \label{example4.6-2}
       $\depth G(I) = 1$.  
  \end{enumerate}
\end{example}
\begin{proof}  (\ref{example4.6-1}) Using CoCoA \cite{cocoa}, the Hilbert series of $I$ is:
\begin{equation*}
    HS_{I}(t)=\frac{19 + 4t + 3t^2 - t^3}{(1-t)^3}.
\end{equation*}
Hence, $e_{2}(I)=0$ and $e_{3}(I)=-1.$

 (\ref{example4.6-2})
 Put $f_1=yz$, $f_2 = xz^2+x^4-y^4$, $f_3=xy^2+xz^2+y^5-z^5$ and   $J = (f_1, f_2, f_3)$. 
 Note that 
  \begin{align*}
\label{ex 1.5 min red} \nonumber
 e_{0}(J) =& e_{0}(yz, x^4-y^4 + xz^2 + yz, y^{5}-z^{5} + xy^2 + xz^2)\\ \nonumber
 =& e_{0}(y, x(x^{3} + z^2), -z^{5} + xz^2)  
 + e_{0}(z, x^4-y^4,  y^{5} + xy^2)\\ \nonumber
=& e_{0}(y, x, z^{5})+  e_{0}(y, x^{3} + z^2,-z^{5} + xz^2 )  
+ e_{0}(z, x^4-y^4, y^{5} + xy^2)\\ \nonumber
=& 5 + e_{0}(y, x^{3},z^2 ) + e_{0}(y, x^{3} + z^2,-z^{3} + x ) 
+  e_{0}(z, x^4, y^2) +  e_{0}(z, x^4-y^4, y^{3} + x)\\
=& 5 + 6 + 2 + 8 + 4 = 25.
\end{align*}
Since, $e_{0}(J)=e_{0}(I)$ and $J\subseteq I,$ thus, $J$ is a minimal reduction of $I.$ Note that $JI^{2}=I^{3}.$ We claim that $\depth G(I)=1.$  Using CoCoA \cite{cocoa}, we get
$$HS_{I/(f_{1})}(t)=\frac{19 + 4t + 3t^2 - t^3}{(1-t)^2}.$$ Hence, from Proposition \ref{hibertcoeffsup} (v), $f_{1}^*$ is a non-zero divisor in $G(I),$ this implies $\depth G(I) \geq 1.$ Suppose $\depth G(I)\geq 2,$ then from \cite[Theorem 2.5]{rv} and the fact $r_{J}(I)=2,$ we have $e_{3}(I)=0,$ which is a contradiction. Therefore, $\depth G(I)=1.$
\end{proof}

In \cite{vv}, authors proved that $\depth G(I^{n}) \geq \depth G(I)$ for all $n \geq 1.$ In the following lemma, we prove a similar result for the associated graded ring of the Ratliff-Rush filtration.

\begin{lemma}  \label{lem1}
    Let $(R,\m)$ be a Cohen-Macaulay local ring with an infinite residue field and $I$ an ideal. Then $\depth \widetilde{G}(I^{k})\geq \depth \widetilde{G}(I)$ for all $k \geq 1.$
\end{lemma}

\begin{proof}
    Let $x_{1},\ldots,x_{s} \in \widetilde{I}\backslash\widetilde{I^{2}}$  be elements such that $x_{1}^{*},\ldots,x_{s}^{*}$ is a $\widetilde{G}(I)$-regular sequence. Then for every $k\geq 1$, $(x_{1}^{*})^{k},\ldots,(x_{s}^{*})^{k}$ is also a $\widetilde{G}(I)$-regular sequence. Note that $(x_{i}^{*})^{k}=(x_{i}^{k})^{*}$ for $i=1,\ldots,s$, therefore, $(x_{1}^{k})^{*},\ldots,(x_{s}^{k})^{*}$ is a $\widetilde{G}(I)$-regular sequence. By Valabrega-Valla \cite{vv}, we have, $x_{1}^{k},\ldots,x_{s}^{k}$ is an $R$-regular sequence and
    \begin{equation*}
  \widetilde{(I^{k})^{n}}\cap(x_{1}^{k},\ldots,x_{s}^{k})
  = \widetilde{I^{kn}}\cap(x_{1}^{k},\ldots,x_{s}^{k})
        =(x_{1}^{k},\ldots,x_{s}^{k}) \widetilde{I^{k(n-1)}} 
        =(x_{1}^{k},\ldots,x_{s}^{k}) \widetilde{(I^{k})^{(n-1)}}\text{ for all } n \geq 1.
    \end{equation*}  
     Therefore, by Valabrega-Valla \cite{vv}, $(x_{1}^{k})^{*},\ldots,(x_{s}^{k})^{*}$ is a $\widetilde{G}(I^{k})$-regular sequence, this implies  $\depth \widetilde{G}(I^{k})\geq \depth \widetilde{G}(I)$ for all $k \geq 1.$
\end{proof}

In the next proposition, we prove an analogous result to Corollary \ref{core2equalszeroCM} for $d$-dimensional Cohen-Macaulay ring to understand the behavior of the depth of the associated graded ring under the assumption that $e_{i}(I)=0$ for all $2 \leq i \leq d$. We also generalize Narita's result \cite[Theorem 1]{narita} to $d$-dimensional Cohen–Macaulay rings. Note that the second part of the following proposition was previously established for maximal Cohen–Macaulay modules in \cite[Theorem 1.1 (3)]{tjp3}, where the author used induction on the dimension of the ring. Our approach is different  we present our proof here.

\begin{proposition} \label{vanishinghilbcoeffCM}
       Let $(R, \m)$ be a Cohen-Macaulay local ring of dimension $d \geq 2,$ and $I$ an $\m$-primary ideal. If $e_{2}(I)=e_{3}(I)=\ldots=e_{d}(I)=0,$ then  the following hold:
       \begin{enumerate}[\normalfont(i)]
           \item \label{mom-1} $\depth G(I) = 0$ or $d.$
           \item \label{mom-2} $G(I^{n})$ is Cohen-Macaulay for sufficiently large $n.$
       \end{enumerate}
\end{proposition}

\begin{proof} 
(\ref{mom-1}) Suppose  $\depth G(I)>0,$ then from \cite[Remark 1.6]{rs}, we have   
$\widetilde{I^{n}}=I^{n}$ for all $n \geq 1.$ By our assumption,  $e_{2}(I)=e_{3}(I)=\ldots=e_{d}(I)=0,$ therefore, 
from \cite[Theorem 6.2]{tjp2}, $\widetilde{G}(I)$ is Cohen-Macaulay. 
By \cite[Theorem 2.5(c)]{rv}, 
we have 
$
 0 =e_2(I) = \widetilde{e}_{2}(I)
=\displaystyle \sum_{n\geq1} n \lambda\left(\frac{\widetilde{I^{n+1}}}{J\widetilde{I^{n}}}\right).$ 
Therefore, $\widetilde{I^{n+1}}= J\widetilde{I^{n}}$ for all $n \geq 1,$ which implies that $I^{n+1}=JI^{n}$ for all 
$n \geq 1.$ Thus, $r_{J}(I) \leq 1$, and by \cite[Theorem 2.9(3)]{rv},   $G(I)$ is Cohen-Macaulay. Therefore,  $\depth G(I) = 0$ or $d.$

(\ref{mom-2})  By our assumption,  $e_{2}(I)=e_{3}(I)=\ldots=e_{d}(I)=0,$ therefore, 
from \cite[Theorem 6.2]{tjp2}, $\widetilde{G}(I)$ is Cohen-Macaulay. Again by Lemma \ref{lem1}, we have  $\depth \widetilde{G}(I^{k})\geq  \depth \widetilde{G}(I)$ for $k \geq 1,$ this implies $\widetilde{G}(I^{k})$ is Cohen-Macaulay. Note that $\widetilde{I^{n}}=I^{n}$ for sufficiently large $n,$ Therefore,  $\widetilde{G}(I^{n}) =G(I^{n})$ for sufficiently large $n,$ this implies $G(I^{n})$ is Cohen-Macaulay for sufficiently large $n.$
\end{proof}

We now give a class of ideals for which $e_{2}(I)=e_{3}(I)=0$ and $\depth G(I)=0.$ The following example is a generalization of \cite[Example 3.8]{rv}.

\begin{example} 
\label{example 3.8}
Let $R=\mathbb{Q}[[x,y,z]]$, $r \ge 2$ and $I=(x^{r}-y^{r},y^{r}-z^{r},xy,yz,zx)$. 
Then the following are true:
\begin{enumerate}[\normalfont(i)]
    \item \label{example 3.8-1}
        $e_{2}(I)=e_{3}(I)=0,$
     \item \label{example 3.8-2}
  $\depth G(I)=0$.
   \end{enumerate}
\end{example}
\begin{proof}
(\ref{example 3.8-1})
Put $K = (x^r, y^r, z^r, xy, xz, yz)$. 
We claim that for all  $n \geq 2$, 
\begin{eqnarray*}
 I^n 
 = K^n
=  \sum_{i=0}^{n}(x^{(n-i)r}, y^{(n-i)r}, z^{(n-i)r}) (xy, xz, yz)^i.
\end{eqnarray*}
Let $n=2$. Clearly $I^2 \subseteq  K^2 = (x^{2r}, y^{2r}, z^{2r}) + (x^{r}, y^r, z^r)(xy, xz, yz)  + (xy, xz,yz)^2.$
Note that
\begin{eqnarray*}
     x^{2r} 
&=& (x^r-y^r)^2 + (x^r-y^r) (y^r-z^r) + (xz)^2 (xz)^{r-2} + (xy)^2 (xy)^{r-2} + (yz)^2(yz)^{r-2}\\
x^r(xy) 
&=& ((x^r-y^r) + (y^r-z^r)) (xy) + (xz)(yz) z^{r-2}\\
x^r (yz) 
&=& (xy)(xz)y^{r-2}\\
x^r (xz)
&=& (x^r-y^r) (xz) + (xy)(yz)y^{r-2}.
\end{eqnarray*}
This implies that $x^r(x^r, xy, xz, yz) \subseteq I^2$. Since the ideal is symmetric in $x,y,z$,  
thus $y^r(x^r, xy, xz, yz) + z^r(x^r, xy, xz, yz) \subseteq I^2$. Hence $I^2 = K^2 = KI$. 
For all $n \geq 3$,  we have
\begin{eqnarray*}
I^n 
= I^2 I^{n-2}
= K^2 I^{n-2}
= K (KI) I^{n-3}
= K^3 I^{n-3}
= \cdots   K^{n-2} I^2 = K^n. 
\end{eqnarray*} 
Thus, $e_{i}(I)=e_{i}(K)$ for all $0 \leq i \leq 3$,
hence, $K \subseteq \widetilde{I}$.  Since $K = (x^r, y,z) \cap (x, y^r, z) \cap (x,y,z^r)$, it is a complete ideal. Therefore, $\widetilde{I}=K$.  
Now, for sufficiently large $n$,   
\begin{align*}
\lambda \left(  \frac{R}{I^n}\right)  
=\lambda\left(  \frac{R}{K^n}\right) &= \binom{2n+2}{3} + (3(r-2))\binom{n+2}{3}\\
&=  8 \binom{n+2}{3} - 4 \binom{n+1}{2} + (3(r-2))\binom{n+2}{3}\\
&= (3r+2) \binom{n+2}{3} - 4 \binom{n+1}{2}.
\end{align*}
Hence, $e_{2}(I)=0$ and $e_{3}(I)=0.$

 (\ref{example 3.8-2}) Clearly, $I \subset K = \widetilde{I}.$ Thus, $I \neq \widetilde{I}$, this implies $\depth G(I)=0.$ 
\end{proof}

We now prove a result concerning the
the signature of $e_{d}(I)$ in  a $d$-dimensional Cohen-
 Macaulay local ring under the assumption that $e_{i}(I)=0$ for all $2 \leq i \leq d-1$  (see Proposition~\ref{edresult}).  This result was 
  proved by Puthenpurakal  in \cite[Theorem 1.4]{tjp3}. Since   our proof has a different approach, we prove the result.  
 We first fix some definitions and prove an elementary lemma.

 \begin{definition} \cite[6.2, 6.3]{tjp}
Let $(R,\m)$ be a Cohen-Macaulay local ring and $I$ be an $\m$-primary ideal. Let $x$ be superficial for $I$,  set $R'=R/(x)$ and $I'=IR'.$ 

We have the following exact sequence of $\mathcal{R}(I)$-modules, called the \textit{second fundamental exact sequence}:
\begin{equation}
\label{eq:u}
    0\longrightarrow \mathcal{B}(x,R) \longrightarrow L^{I}(R)(-1
) \xlongrightarrow{\psi_{x}} L^{I}(R) \longrightarrow L^{I'}(R') \longrightarrow 0,
\end{equation}
 where $\psi_{x}$ is multiplication by $x,$ $\mathcal{B}(x,R)= \displaystyle \bigoplus_{n \geq 0} \frac{(I^{n+1}:x)}{I^{n}},$ and $L^{I}(R)=\displaystyle \bigoplus_{n \geq 0} R/I^{n+1}.$  See \cite{tjp,tjp2} for references and further applications of $L^{I}(R)$ module.  
 Note that, $(I^{n+1}:x)=I^{n}$ for sufficiently large $n.$ Thus, $\mathcal{B}(x,R)$ has finite length, therefore, $H^{0}_{\M}\left(\mathcal{B}(x,R)\right)$ $=\mathcal{B}(x,R),$ where $H^{i}_{\mathfrak{M}}(\text{\textunderscore})$ denotes that $i$-th local cohomology with support in $\M=\m \oplus \mathcal{R}(I)_{+}=\m \oplus\displaystyle \left(\bigoplus_{n \geq 1}I^{n}t^{n}\right).$
 The sequence in (\ref{eq:u}) induces a long exact sequence of local cohomology modules. 
\begin{equation} \label{4}
    \begin{aligned}
        0 \longrightarrow \mathcal{B}(x,R)  &  \longrightarrow H^{0}_{\M}\left(L^{I}(R)\right)(-1) \longrightarrow H^{0}_{\M}\left(L^{I}(R)\right) \longrightarrow H^{0}_{\M}\left(L^{I'}(R')\right) \\
     & \xlongrightarrow{\delta} H^{1}_{\M}\left(L^{I}(R)\right)(-1) \longrightarrow H^{1}_{\M}\left(L^{I}(R)\right) \longrightarrow H^{1}_{\M}\left(L^{I'}(R')\right) \cdots.
    \end{aligned}
\end{equation}
\end{definition}

In the following lemma, we compare $\lambda(\mathcal{B}(x,R))$ and $\lambda\left( {\mathcal H}^{0}_{\M}\left(L^{I'}(R')\right)\right).$

\begin{lemma} \label{comparing}
Let $(R, \m)$ be a Cohen-Macaulay local ring and $I$ be an $\m$-primary ideal. 
 Then 
 \begin{equation*}
     \lambda(\mathcal{B}(x,R))\leq \lambda\left(H^{0}_{\M}\left(L^{I'}(R')\right)\right).
 \end{equation*}
\end{lemma}

\begin{proof}
    From the long exact sequence of local cohomology modules in (\ref{4}), we have the following exact sequence:
    \begin{equation} \label{22}
       0 \longrightarrow \mathcal{B}(x,R)   \longrightarrow H^{0}_{\M}\left(L^{I}(R)\right)(-1) 
        \longrightarrow H^{0}_{\M}\left(L^{I}(R)\right) 
        \longrightarrow H^{0}_{\M}\left(L^{I'}(R')\right) 
        \longrightarrow C 
        \longrightarrow 0,
    \end{equation}

    where $C:=\textrm{image} (\delta)$. 
Taking  lengths of modules in (\ref{22}), we get
\begin{align*}
       \lambda\left(\mathcal{B}(x,R)\right)
 & = \lambda\left(H^{0}_{\M}\left(L^{I}(R)\right)(-1)\right)
  -\lambda\left(H^{0}_{\M}\left(L^{I}(R)\right)\right)
  +\lambda\left(H^{0}_{\M}\left(L^{I'}(R')\right)\right)
  - \lambda(C)\\
 & \leq \lambda\left(H^{0}_{\M}\left(L^{I'}(R')\right)\right)  \text{ (as  $\lambda\left(H^{0}_{\M}\left(L^{I}(R)\right)\right) < \infty)$}. \quad \quad\qedhere
\end{align*} 
\end{proof}

 \begin{proposition} \label{edresult}
     Let $(R, \m)$ be a Cohen-Macaulay local ring of dimension $d \geq 2,$ and $I$ an $\m$-primary ideal. If $e_{2}(I)=e_{3}(I)=\ldots=e_{d-1}(I)=0,$ then $(-1)^{d}e_{d}(I) \geq 0.$
\end{proposition}

\begin{proof}
    Let $x \in I \backslash I^{2}$ be a superficial element 
    for $I,$ which is also a non zero-divisor in $R.$  Set $R'=R/(x)$ and $I'=IR'.$ Then $R'$ is a Cohen-
    Macaulay local ring of dimension $d-1$, and from  Proposition \ref{hibertcoeffsup},
    $e_{i}(I')=e_{i}(I)$ for all $i=0,1,\ldots, d-1.$ 
    Since $e_{2}(I')=e_{3}(I')=\ldots=e_{d-1}(I')=0,$ therefore, 
    from Proposition \ref{vanishinghilbcoeffCM}, $\depth G(I')=0$ or $d-1.$ 
     If $\depth G(I')=d-1,$ then by Sally's descent $\depth G(I)=d.$ 
    Thus, from \cite[Corollary 2(ii)]{marley}, we have $e_{d}(I) = 0.$
Suppose that $\depth G(I')=0.$ Since 
    $e_{2}(I')=e_{3}(I')=\ldots=e_{d-1}(I')=0,$ therefore, from \cite[Theorem 6.2]{tjp2}, $\widetilde{G}(I')$ is Cohen-Macaulay with minimal multiplicity.
   
 By \cite [Theorem 2.5]{rv}, we have
\begin{equation*}
    \widetilde{e}_{2}(I')
         =\displaystyle \sum_{n \geq 1}  n\lambda \left(\frac{ \widetilde{I'^{n+1} }}{L\widetilde{I'^{n}}}\right)=0.
\end{equation*}

    This implies that $\widetilde{I'^{n+1}} = L \widetilde{I'^{n}}$ for all $n \geq 1$, where $L$ is a minimal reduction of $I'$. Hence it follows from \cite [Theorem 2.5]{rv}, 
\begin{equation*}
    \widetilde{e}_{d}(I')
         =\displaystyle \sum_{n \geq d-1}  \binom{n}{d-1}\lambda \left(\frac{ \widetilde{I'^{n+1} }}{L\widetilde{I'^{n}}}\right)=0.
\end{equation*}
    Again by \cite[1.5(b)]{tjp}, we have
 \begin{equation} \label{111}
        e_{d}(I')
        =\widetilde{e}_{d}(I')+(-1)^{d}\displaystyle
            \sum_{n \geq 0}\lambda\left(\frac{\widetilde{I'^{n+1}}}{I'^{n+1}}\right)
        =(-1)^{d}\displaystyle\sum_{n \geq 0}\lambda\left(\frac{\widetilde{I'^{n+1}}}{I'^{n+1}}\right)=(-1)^{d}\lambda\left(H^{0}_{\M}\left(L^{I'}(R')\right)\right).
    \end{equation}

   And, by Proposition \ref{hibertcoeffsup} (iv), we have 
    \begin{equation} \label{222}
        e_{d}(I')
        =e_{d}(I)+(-1)^{d}\displaystyle\sum_{n \geq 0}\lambda(\mathcal{B}(x,R)). 
    \end{equation}
    Therefore, from equations (\ref{111}), (\ref{222}) and Lemma \ref{comparing}, 
    we have \[(-1)^{d}e_{d}(I)=\lambda\left(H^{0}_{\M}\left(L^{I'}(R')\right)\right)- \lambda(\mathcal{B}(x,R)) \geq 0.\qedhere\] 
\end{proof}

We conclude this section with an example which illustrates Proposition \ref{edresult}.

\begin{example}
\label{example e4=1}
Let $R=\mathbb{Q}[[x,y,z,t]]$ where $x,y,z,t$ are variables. For any integer $r \geq 2,$ let $I = (x^r-y^r, y^r-z^r, xy, xz, yz, t)$, then  $e_2(I)=e_3(I)=0$ and $e_{4}(I)=1.$
 \end{example}
\begin{proof} 
   We need to show that $(t) \cap I^n = (t)I^{n-1}$  for all $n \geq 1$. Write $I = J + (t),$ where $J = (x^r-y^r, y^r-z^r, xy, xz, yz)$. Then for any $n \geq 1,$ we have 
$I^n = J^n + (t) J^{n-1} + \cdots (t^{n-1})J + (t^n)$. Clearly, $(t) I^{n-1} \subseteq (t) \cap I^n.$  Let $a \in (t) \cap I^n$, then $a = bt,$ and $a = bt = c_{n} + c_{n-1}t + \cdots c_1 t^{n-1} + c_0 t^n,$ where $c_i \in J^i$ for all $i=0,\ldots n $. 
Hence $c_n \in (t) \cap J^n = t(J^n:t) = t J^n $. This implies that $b \in I^{n-1}$. Therefore, $a \in (t) I^{n-1},$ this implies $(t) \cap I^n = (t)I^{n-1}$ for all $n \geq 1.$ Thus, $\depth G(I) \geq 1.$ From Example \ref{example 3.8}, and Sally's descent, we have $\depth G(I)=1$. Since $t^{\star}$ is a non zero-divisor in $G(I)$, thus, we have:
\begin{eqnarray*}  
\lambda \left( \frac{R}{I^n} \right) &=&\sum_{j=1}^n \lambda \left( \frac{I^{j-1}}{I^j} \right) \\
&=& \sum_{j=1}^n\lambda \left( \frac{R}{I^j + (t)} \right) \text{(by Singh's formula \cite{singh})} \\ 
&=& \sum_{j=2}^n \left[  (3r+2) \binom{j+1}{2} - 4 \binom{j+2}{3}\right] + \lambda\left(\frac{R}{I}\right)  \mbox{ (from Example \ref{example 3.8})}\\ 
&=& \sum_{j=1}^n \left[  (3r+2) \binom{j+1}{2} - 4 \binom{j+2}{3}\right] + \lambda\left(\frac{R}{I}\right)-(3r-2)\\
&=& (3r+2) \binom{n+3}{4} - 4 \binom{n+2}{3} + 1.
\end{eqnarray*}
Hence, $e_{2}(I)=e_{3}(I)=0$ and $e_{4}(I)=1.$
\end{proof}

\section*{Acknowledgement}
The first author is partially supported by a grant from Infosys Foundation and ANRF  MATRICS grant (File no. MTR/2023/000661). 
 The third author acknowledges the support of the Government of India through the Prime Minister's Research Fellowship during the course of this work.

\providecommand{\bysame}{\leavevmode\hbox to3em{\hrulefill}\thinspace}
\providecommand{\MR}{\relax\ifhmode\unskip\space\fi MR }
 %\MRhref is called by the amsart/book/proc definition of \MR.
\providecommand{\MRhref}[2]{
  \href{http://www.ams.org/mathscinet-getitem?mr=#1}{#2}
}

\end{document}